\documentclass[letterpaper]{article}
\usepackage[utf8]{inputenc}
\usepackage{indentfirst}
\usepackage{amsthm}
\usepackage{graphicx}
\usepackage{float}
\usepackage{harmony}
\usepackage{amssymb,quiver}
\usepackage{enumitem}
\usepackage{stmaryrd}
\usepackage{mathtools}
\usepackage[mathscr]{euscript}
\usepackage{tikz}
\usepackage{algpseudocode}
\usepackage{tikz-cd}
\usepackage{hyperref}
\hypersetup{
    colorlinks = true,
    citecolor = orange,
    linkcolor = blue,
}
\usepackage{listings}
\usetikzlibrary{arrows.meta}
\newcommand{\R}{\Re}
\usepackage[margin=1in]{geometry}
\usepackage{multicol}
\usepackage{blkarray}
\usepackage{tabularray}
\usepackage{colortbl}

\theoremstyle{plain}
\newtheorem{theorem}{Theorem}
\newtheorem{lemma}{Lemma}
\newtheorem{proposition}{Proposition}
\newtheorem{corollary}{Corollary}

\theoremstyle{definition}
\newtheorem{definition}{Definition}

\newtheorem{remark}{Remark}

\newtheorem{example}{Example}
\newtheorem{notation}{Notation}

\theoremstyle{remark}

\def\R{\mathbb{R}}
\def\N{\mathbb{N}}

\def\C{\mathbb{C}}

\def\tr{\mathrm{tr}}

\def\det{\mathrm{det}}

\title{Local Statistics of the $M_n$-Dimer Model}
\author{Nickolas Anderson
    \and Moriah Elkin
    \and Elizabeth Kelley
    \and Nicholas Ovenhouse
    \and Kayla Wright}
\date{}

\begin{document}

\maketitle

\abstract{
    The classical dimer model is concerned with the (weighted) enumeration of perfect matchings of a graph. 
    An $n$-dimer cover is a multiset of edges that can be realized as the disjoint union of $n$ individual matchings.
    For a probability measure recently defined by Douglas, Kenyon, and Shi, which we call the $M_n$ dimer model,
    we study random $n$-dimer covers on bipartite graphs with matrix edge weights and produce formulas for local edge statistics and correlations.
    We also classify local moves that can be used to simplify the analysis of such graphs.
}

\tableofcontents

\section{Introduction}

The \emph{dimer model} is the study of random perfect matchings (also called \emph{dimer covers}) of graphs. It is a very well-studied subject
at the intersection of combinatorics, probability, and statistical mechanics (see \cite{kenyon_survey} for a survey). There are many interesting
connections to cluster algebras \cite{msw_13}, total positivity \cite{psw_09}, and integrable systems \cite{gk_13}.
The dimer model is an example of a \emph{determinantal point process} \cite{kenyon_97}, which means the probability of any given subset of edges
appearing in a random matching is given by a determinant. Thus local statistics and correlations are easily computable.

In this article, we study \emph{$n$-dimer covers}, which are multisets of edges so that every vertex is incident to exactly $n$ edges (counted with multiplicity).
An $n$-dimer cover can be obtained by choosing $n$ perfect matchings and superimposing them. 
The simplest probability measure on the set of $n$-dimer covers
is the product measure on $n$ independent samples from the dimer model. In \cite{dks_24}, Douglas, Kenyon, and Shi defined a much more general 
probability measure on the set of $n$-dimer covers which includes this product measure as a special case.
The measure depends on a choice of $n \times n$ matrix for each edge of the
graph, where the product measure corresponds to choosing all diagonal matrices. We will call this measure the \emph{$M_n$-dimer model} to reflect the
dependence on the choice of matrix edge weights.

The $M_n$-dimer model has connections to representation theory and the geometry of character varieties. The weight (or probability) of each $n$-dimer cover
is closely related to Kuperberg's web invariants \cite{kuperberg} \cite{fll_19}. These weights can also be interpreted as regular functions on the
character variety (moduli space of representations of the fundamental group) of the graph. Here we will use the equivalent, more combinatorial definition given in \cite{dks_24}
of the weights and probabilities.

The aim of the present paper is to generalize some of the known determinantal formulas of the dimer model to the case of the $M_n$-dimer model.
In particular, we study the edge multiplicities of random $n$-dimer covers, giving expressions for probability generating functions, expectation values,
and correlations between multiple edges. In all cases, our formulas can be expressed as traces of noncommutative matrix-valued versions of the formulas
from the ordinary single-dimer model.

The structure of the paper is as follows. In Section 2, we give some more specific background and definitions. In section 3, we define the main ingredient
of our constructions: for each edge $e$, we define a matrix $P_e$, which should be thought of as a generalization of the edge probability
(indeed when $n=1$ the two coincide). These matrices are the building blocks of all our later formulas. Section 4 discusses several elementary
transformations of graphs which are well-known to experts. These transformations preserve the partition functions of the ordinary single-dimer model,
and we discuss how they can be upgraded to the $M_n$-dimer model. Sections 5 and 6 contain the main results of the paper, giving formulas for 
the distribution of the edge multiplicities, and correlations between several edges. Section 7 gives some example applications of our results.
Section 8 explains how to generalize our results to the case of \emph{mixed dimer covers} (where each vertex can have a different degree/multiplicity),
along with some example applications. 

\section{Background}\label{section:background}

\subsection {The Dimer Model}
Let $G = (V,E)$ be a bipartite planar graph.
\begin {definition}
     An \emph{$n$-dimer cover} on $G$ is a multiset of edges such that each vertex is incident to $n$ edges in the multiset (counted with multiplicity). Equivalently, it is a function $m \colon E \to \N$ such that $\sum_{e \sim v} m(e) = n$ for all $v \in V$.
    We call $m(e)$ the \emph{multiplicity} of edge $e$, and we will often write $m_e$ instead of $m(e)$. We denote the set of all $n$-dimer covers by $\Omega_n(G)$.
\end {definition}

When $n=1$, a dimer cover is the same thing as a perfect matching, and by a result of Kasteleyn \cite{kasteleyn}, the perfect matchings of a planar graph can be enumerated with a Pfaffian or determinant.
We will begin by describing this classic result, and then will state a more general version for the $M_n$-dimer model in Section~\ref{sec:Mn_model}.

Let $V = B \sqcup W$ be a bipartition of the vertices of $G$ into black and white.
    Also let $\mathrm{wt} \colon E \to \Bbb{R}_+$ be a positive weight function on the edges $E$ of $G$.
\begin {definition}
    The weight of a perfect matching $M \in \Omega_1(G)$ is defined as $\mathrm{wt}(M) = \prod_{e \in M} \mathrm{wt}(e)$.
\end {definition}

\begin {definition}
    A \emph{Kasteleyn connection} is an edge weighting $\varepsilon \colon E \to \{\pm 1\}$ such that for each face $f$ with $2\ell$ edges,
    we have $\prod_{e \in f} \varepsilon(e) = (-1)^{\ell-1}$.
    We then define $K$ (the \emph{Kasteleyn matrix}) to be the $|W| \times |B|$ matrix whose $(w,b)$-entry is $\varepsilon(w,b)\mathrm{wt}(w,b)$ if
    $w$ and $b$ are adjacent vertices (and 0 if they are not adjacent).
\end {definition}

\begin {theorem}\cite{kasteleyn} \label{kasteleyn_theorem}
    For a bipartite planar graph $G$, with Kasteleyn matrix $K$,
    \[ \left| \det(K) \right| = \sum_{M \in \Omega_1(G)} \mathrm{wt}(M). \]
\end {theorem}

\begin{definition}
The \emph{dimer model} on $G$ is the probability measure on the set $\Omega_1(G)$ of perfect matchings such that
each matching $M \in \Omega_1(G)$ has probability $\frac{\mathrm{wt}(M)}{Z}$, and $Z = \left| \det(K) \right| = \sum_M \mathrm{wt}(M)$.
The normalization constant $Z$ is called the \emph{partition function} of the model.
\end{definition}

The main statistical observables in this model are the edge probabilities. For each edge $e \in E$, let $p_e = \mathrm{Pr}[e \in M]$
be the probability that the edge $e$ appears in a randomly chosen perfect matching. Kenyon observed that these probabilities are essentially
the entries of the matrix $K^{-1}$, as stated in Theorem \ref{thm:kenyon} below. 

\begin {notation}
    We will write the entries of an inverse matrix using superscripts: $K^{ij} := (K^{-1})_{ij}$.
\end {notation}

\begin {notation}
    Let $M$ be an $N \times N$ matrix, and let $I$ and $J$ be two subsets of $\{1,2,\dots,N\}$ of equal size.
    We will write $\Delta_{I,J}(M)$ for the minor with row set $I$ and column set $J$.
\end {notation}

\begin {theorem} \label{thm:kenyon} \cite{kenyon_97} \\
    Suppose the white and black vertices of $G$ are labeled $w_1,\dots,w_N$ and $b_1,\dots,b_N$, and let $K$ be the Kasteleyn matrix.
    \begin {enumerate}
        \item[(a)] For an edge $e = (w_i,b_j)$, we have $p_e = K^{ji} K_{ij}$.
        \item[(b)] For a collection of edges $e_1,\dots,e_k$ with $e_t = (w_{i_t},b_{j_t})$, the probability $\mathrm{Pr}[e_1,\dots,e_k \in M]$ that
                   all $k$ edges are in a randomly chosen perfect matching is equal to $\Delta_{J,I}(K^{-1}) K_{i_1,j_1} \cdots K_{i_k,j_k}$.  
    \end {enumerate}
\end {theorem}

Let us briefly mention a simple corollary of this theorem. The \emph{covariance} of two random variables $X$ and $Y$ 
is $\mathrm{Cov}(X,Y) := \Bbb{E}(XY) - \Bbb{E}(X) \Bbb{E}(Y)$. For an edge $e$, let $X_e$ be the random variable which is the
indicator function for the event of $e$ appearing in a matching. 
Then $\Bbb{E}(X_e) = p_e$. From Theorem \ref{thm:kenyon}, we quickly get the following.

\begin {corollary}
    Let $e_1 = (w_{i_1},b_{j_1})$ and $e_2 = (w_{i_2},b_{j_2})$ be two different edges of $G$. The covariance of $X_{e_1}$ and $X_{e_2}$ is
    \[ \mathrm{Cov}(X_{e_1},X_{e_2}) = - K^{j_1 i_2} K^{j_2 i_1} K_{i_1 j_1} K_{i_2 j_2} \]
\end {corollary}
On the other hand, when $e_1 = e_2$, then this instead is the variance, which for an indicator function is simply $\mathrm{Var}(X_e) = p_e(1-p_e)$.

In this paper, we will study a generalization of the dimer model that is a probability measure on $\Omega_n(G)$, the set of $n$-dimer covers.
The analog of the random variable $X_e$ will be the random variable $m_e$ (the edge multiplicity), which can take
values in $\{0,1,2,\dots,n\}$. Our main results will be various generalizations of the above theorems and formulas. 
We will associate to each edge $e \in E$ a matrix $P_e$, generalizing the edge probability $p_e$, 
and see that the statistical observables are given by functions of the eigenvalues of the $P_e$ matrices that
are analogous to the above formulas.

\subsection {The $M_n$-Dimer Model} \label{sec:Mn_model}

Recall that the dimer model on $\Omega_1(G)$ is a probability measure depending on an edge weighting $\mathrm{wt} \colon E \to \R_+$.
More generally, we will discuss a probability measure on $\Omega_n(G)$, the set of $n$-dimer covers, first given in \cite{dks_24}. We will
call this measure the \emph{$M_n$ dimer model}.
To define this measure, we require two pieces of input: matrix edge weights and cilia. The edge weights are $n \times n$ matrices given by
a function $\mathrm{wt} \colon E \to \mathrm{Mat}_n(\R)$. Also, at each vertex we draw a \emph{cilium} (a small line segment
coming out of the vertex) that gives a linear order of the incident half-edges. This order is defined at a black (resp. white) vertex by starting at the cilium and proceeding counter-clockwise (resp. clockwise).

For each $n$-dimer cover $\omega \in \Omega_n(G)$, its weight $\mathrm{wt}(\omega)$ was defined in \cite{dks_24}
to be the trace (i.e. contraction) of some tensor involving the matrix edge weights of $G$.
We will instead give a more combinatorial definition. Define a \emph{half-edge coloring} of an $n$-dimer cover $\omega \in \Omega_n(G)$
to be a labeling of the half-edges of $\omega$ such that around each vertex, the numbers $1,2,\dots,n$ each appear once. 

For a coloring $c$, there is a permutation $\sigma_v(c)$ at each vertex $v$ obtained by reading the colors in the order determined by the cilia.
For edges with multiplicity, we read the colors on that half edge in their natural (increasing) order.
The \emph{sign} of the coloring is defined as the product of the signs of all these permutations (over all the vertices). We denote the sign of
a coloring $c$ by $(-1)^c$. 

\begin {definition}\label{def:dimerwt} \cite{dks_24}
    Let $G$ be a planar bipartite graph, with a choice of matrix edge weights and cilia. For an $n$-dimer cover $\omega \in \Omega_n(G)$, 
    its \emph{weight} is defined as
    \[ \mathrm{wt}(\omega ) = \sum_{\substack{\mathrm{coloring} \\ c}} (-1)^c \prod_{e \in E} \Delta_{I^c_e,J^c_e}(\mathrm{wt}(e)), \]
    where $I^c_e$ and $J^c_e$ are the sets of colors assigned to the half edges near the white and black ends of the edge $e$.
\end {definition}

\begin {definition}
    Supposing the matrix edge weights are chosen such that $\mathrm{wt}(\omega) > 0$ for every $n$-dimer cover $\omega$, we define the \emph{$M_n$-dimer model}
    to be the probability measure on $\Omega_n(G)$ where the probability of $\omega$ is proportional to $\mathrm{wt}(\omega)$. That is, if
    $Z = Z(G)= \sum_\omega \mathrm{wt}(\omega)$ is the \emph{partition function}, then $\mathrm{Pr}[\omega] = \frac{1}{Z} \mathrm{wt}(\omega)$.
\end {definition}

\begin {example} \label{ex:dimerwt}
Let $G$ be the planar bipartite graph below with $3\times 3$ matrix edge weights as shown. All cilia are oriented outward. 
Let $\omega$ be the 3-dimer cover of $G$ shown to the right.

        
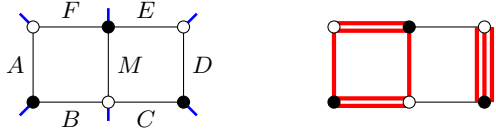
\begin {center}
    \begin {tikzpicture}
        \draw (0,0) -- (1,0) -- (2,0) -- (2,1) -- (1,1) -- (0,1) -- cycle;
        \draw (1,0) -- (1,1);

        \draw[blue, line width=1] (0,0) --++ (-135:0.25);
        \draw[blue, line width=1] (1,0) --++ (-90:0.25);
        \draw[blue, line width=1] (2,0) --++ (-45:0.25);
        \draw[blue, line width=1] (2,1) --++ (45:0.25);
        \draw[blue, line width=1] (1,1) --++ (90:0.25);
        \draw[blue, line width=1] (0,1) --++ (135:0.25);
        
        \draw[fill=black] (0,0) circle (0.08);
        \draw[fill=black] (1,1) circle (0.08);
        \draw[fill=black] (2,0) circle (0.08);
        \draw[fill=white] (1,0) circle (0.08);
        \draw[fill=white] (0,1) circle (0.08);
        \draw[fill=white] (2,1) circle (0.08);

        \draw (0,0.5) node[left] {$A$};
        \draw (0.5,0) node[below] {$B$};
        \draw (1,0.5) node[right] {$M$};
        \draw (0.5,1) node[above] {$F$};
        \draw (1.5,0) node[below] {$C$};
        \draw (2,0.5) node[right] {$D$};
        \draw (1.5,1) node[above] {$E$};

        \begin {scope}[shift={(4,0)}]
        \draw (1,1) -- (2,1);
        \draw (1,0) -- (2,0);
        \draw[red, ultra thick] (0,0) -- (0,1);
        \draw[red, ultra thick] (1,0) -- (1,1);
        \draw[red, ultra thick] (0,-0.05) -- (1,-0.05);
        \draw[red, ultra thick] (0,.05) -- (1,0.05);
        \draw[red, ultra thick] (0,0.95) -- (1,0.95);
        \draw[red, ultra thick] (0,1.05) -- (1,1.05);
        \draw[red, ultra thick] (1.9,0) -- (1.9,1);
        \draw[red, ultra thick] (2,0) -- (2,1);
        \draw[red, ultra thick] (2.1,0) -- (2.1,1);
        
        \draw[fill=black] (0,0) circle (0.08);
        \draw[fill=black] (1,1) circle (0.08);
        \draw[fill=black] (2,0) circle (0.08);
        \draw[fill=white] (1,0) circle (0.08);
        \draw[fill=white] (0,1) circle (0.08);
        \draw[fill=white] (2,1) circle (0.08);
        \end {scope}
    \end {tikzpicture}
    \end {center}
Let us analyze the half-edge colorings of $\omega$. Observe that for the two rightmost vertices, there is only one half-edge coloring of $\omega$, and the corresponding permutation at either vertex is $\sigma_v(c) = 123$. 
We can thus see from Definition \ref{def:dimerwt} that the triple edge contributes a factor of $\det(D)$ to $\mathrm{wt}(\omega)$.

For each of the remaining four vertices, there are three possible half-edge colorings of $\omega$, resulting in the permutations $\sigma_v(c) \in \{123, 132, 231\}$. Thus, we have 81 such half-edge colorings of $\omega$, 40 of which have sign $-1$ and 41 of which have sign $+1$. The weight of $\omega$ is then $\det(D) \cdot X$,
where $X$ is a signed sum of 81 terms, each of which is a product of $1 \times1$ and $2 \times 2$ minors of the matrices $A,B,F$, and $M$.

In the next section, we will return to this example after seeing a way to simplify the calculation.

\end {example}

\begin {remark}
    Choosing matrix edge weights so that $\mathrm{wt}(\omega) > 0$ for all $\omega$ is not necessarily an easy task, and the question of which choices
    of matrices satisfy this condition is subtle and difficult. The rest of this article will not address this question,
    assuming that matrix edge weights have been chosen for which all weights are positive, but we will now briefly mention some cases where this is known to be true.
    \begin {itemize}
        \item The signs appearing in the definition of $\mathrm{wt}(\omega)$ depend on the choice of cilia. \cite{dks_24} proved that
              if every face has an even number of inward-pointing cilia (they referred to such a choice as \emph{positive cilia}), 
              and $\mathrm{wt}(e) = \mathrm{Id}_n$ for every edge $e \in E$, 
              then we will have $\mathrm{wt}(\omega) > 0$ always (in fact $\mathrm{wt}(\omega)$ will be the number of proper edge-colorings of $\omega$).
              Therefore by continuity, if we have positive cilia, and if $\mathrm{wt}(e)$ is sufficiently close to the identity matrix for every edge $e$,
              then all weights will be positive.
        \item In \cite[Theorem 5]{ko_23}, a parameterization is given of a large subset of the space of all matrix-valued edge labelings for which all $n$-dimer covers
              have positive weights. 
    \end {itemize}
\end {remark}

A version of Kasteleyn's theorem is still true for the $M_n$-dimer model, which we will now explain (see \cite{dks_24}, \cite{ko_23} for more
details and proofs).
Given matrix edge weights and choice of cilia, we define a \emph{Kasteleyn connection} $\varepsilon \colon E \to \{\pm 1\}$ 
by declaring that for each face $f$
with $2\ell$ edges and $k$ cilia pointing into the interior of the face, we have that $\prod_{e \in f} \varepsilon(e)$ is equal to
$(-1)^{\ell-1+k}$ if $n$ is even, and $(-1)^{\ell-1}$ if $n$ is odd.
We now define an analogous version of the Kasteleyn matrix. It is an $n|W| \times n|B|$ matrix, which we define via its $n \times n$ blocks.

\begin {notation}
    Let $K_{[i],[j]}$ denote the $n \times n$ block of $K$ in rows $(i-1)n+1, \dots, in$ and columns $(j-1)n+1, \dots, jn$. 
    In other words, it is the $(i,j)$-entry, if we think of $K$ as a $|W| \times |B|$ matrix whose entries are themselves $n \times n$ matrices.
    We will also use a version of this notation for blocks of the inverse matrix: $K^{[i],[j]} := (K^{-1})_{[i],[j]}$.
\end {notation}

For each edge $e = (w,b)$, define $K_{[w],[b]} := \varepsilon(e)\mathrm{wt}(e)$, and $K_{[w],[b]} = 0$ if $(w,b) \not \in E$.

\begin {theorem} \cite[Thm 4.1]{dks_24} \\
    Let $G$ be a planar bipartite ciliated graph, with a choice of matrix-valued edge weights, and associated Kasteleyn matrix $K$. Then
    \[ \left| \det(K) \right| = \sum_{\omega \in \Omega_n(G)} \mathrm{wt}(\omega). \]
\end {theorem}

\subsection {Schur Reduction}

In this subsection, we briefly state two useful identities 
that will be used throughout the paper. Both identities use the notion of the \emph{Schur complement}.

\begin {definition}
    Let $M = \begin{pmatrix} A & B \\ C & D \end{pmatrix}$ be a block matrix, with $A$ and $D$ square matrices.
    If $D$ is invertible, the \emph{Schur complement} of $M$ with respect to the block $D$ is defined as
    $M/D := A - BD^{-1}C$.
    Similarly, the Schur complement with respect to $A$ is $M/A := D - CA^{-1}B$.
\end {definition}

For the purposes of this article, there are two particularly useful applications of the Schur complement. The first is a way to express
determinants of block matrices in terms of determinants of smaller matrices.

\begin {theorem} [Schur Reduction Formula] \cite{schur_17}
    Let $M = \begin{pmatrix} A&B\\C&D \end{pmatrix}$ be a block matrix with $A$ and $D$ square. Then
    \[ \det(M) = \det(D) \det(M/D). \]
\end {theorem}
The second application expresses the inverse of a block matrix again as a block matrix. 
\begin {theorem}
    Let $M = \begin{pmatrix} A&B\\C&D \end{pmatrix}$ be a block matrix with $A$ and $D$ square. Then
    \[ M^{-1} = \begin{pmatrix} (M/D)^{-1} & \ast \\ \ast & \ast \end{pmatrix}. \]
\end {theorem}

Lastly, we note that one can similarly define Schur complements with respect to any square submatrix, with analogous
versions of the two theorems above. To do so, one simply applies permutation matrices to put the desired block
in the upper-left corner, and then applies the formulas above. These more general Schur complements are the
(inverses of the) \emph{quasideterminants} of \cite{quasideterminants}.

\section {The Probability Matrix}\label{sec:pmat}
\label{section:probabilityMatrix}

\begin {definition}
    Let $G$ be a planar bipartite graph with matrix-valued edge weights and Kasteleyn matrix $K$.
    Denote the vertices by $w_1,\dots,w_N$ and $b_1,\dots,b_N$, and let $e = (w_i,b_j)$ be any edge. 
    Define the \emph{probability matrix} of the edge $e$ to be $P_e := K^{[j],[i]}K_{[i],[j]}$.
\end {definition}

\begin {example}
    Consider the graph pictured in Figure \ref{fig:small_example}, which has matrix edge weights $A,B,C,D$. The inverse of the Kasteleyn matrix is given by
    \[ K^{-1} = \begin{pmatrix} (A+DC^{-1}B)^{-1} & (B+CD^{-1}A)^{-1} \\ -(D+AB^{-1}C)^{-1} & (C+BA^{-1}D)^{-1} \end{pmatrix} \]
    The probability matrix for the left vertical edge (with matrix edge weight $A$) is then
    \[ P_e = K^{[1],[1]}K_{[1],[1]} = (A+DC^{-1}B)^{-1}A = (I+A^{-1}DC^{-1}B)^{-1}. \]
    \begin {figure}
    \centering
    \begin {tikzpicture}
        \draw (0,0) -- (1,0) -- (1,1) -- (0,1) -- cycle;
        \draw[fill=black] (0,0) circle (0.1);
        \draw[fill=black] (1,1) circle (0.1);
        \draw[fill=white] (1,0) circle (0.1);
        \draw[fill=white] (0,1) circle (0.1);

        \draw (0,0.5) node[left] {$A$};
        \draw (0.5,0) node[below] {$B$};
        \draw (1,0.5) node[right] {$C$};
        \draw (0.5,1) node[above] {$D$};

        \draw (4,0.5) node {$K = \begin{pmatrix} A & -D \\ B & C \end{pmatrix}$};
    \end {tikzpicture}
    \caption{An example of a graph with matrix edge weights, and its Kasteleyn matrix.}
    \label{fig:small_example}
    \end {figure}
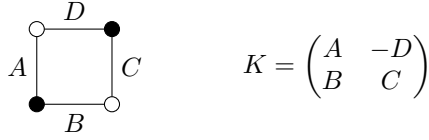
\end {example}

\begin {remark}
    When $n=1$, Theorem \ref{thm:kenyon} says that $P_e$ (which is a scalar in this case)
    is the probability that the given edge appears in a dimer cover. The matrix $P_e$ is thus a natural generalization of the edge probabilities.
    We will see in Section \ref{sec:formulas} that local edge statistics can be written in terms of this matrix.
\end {remark}

\begin {definition}
    Let $G$ be a graph, and $\mathrm{wt} \colon E \to \mathrm{Mat}_n(\Bbb{R})$ be a matrix-valued edge weighting.
    For any black vertex, and any $M \in \mathrm{GL}_n(\Bbb{R})$, define the corresponding local \emph{gauge transformation} to
    be the operation which changes the weights of all incident edges by right-multiplication with $M$; that is, $\mathrm{wt}(e) \mapsto \mathrm{wt}(e) \cdot M$.
    Similarly, at a white vertex, define a gauge transformation by left-multiplication: $\mathrm{wt}(e) \mapsto M \cdot \mathrm{wt}(e)$.
    The \emph{gauge group} is the group generated by all such transformations, ranging over all vertices and all $M \in \mathrm{GL}_n(\Bbb{R})$.
\end {definition}

The reason for considering such transformations is that the probability measure on $\Omega_n(G)$ does not depend on the actual edge matrices,
but only on the gauge equivalence class.

\begin {proposition} \cite{dks_24, ko_23} \label{prop:gauge}
    Two gauge-equivalent edge-weightings of $G$ give the same probability measure on $\Omega_n(G)$. 
\end {proposition}
\begin {proof}
    Gauge transformations change the Kasteleyn matrix by left or right multiplication with a block-diagonal matrix, whose
    diagonal blocks are the matrices $M$ of the local gauge transformations. Thus $\det(K)$ changes by an overall factor of $\det(M)$.
    By introducing formal parameters on each edge, one can see that in fact the weight of every individual $n$-dimer cover
    is multiplied by the same factor $\det(M)$, and so the resulting probability measure is unchanged.
\end {proof}

Helpfully, the edge probability matrices $P_e$ are essentially invariant under the gauge action.

\begin {proposition}
    Let $\mathrm{wt} \colon E \to \mathrm{Mat}_n(\Bbb{R})$ and $\mathrm{wt}' \colon E \to \mathrm{Mat}_n(\Bbb{R})$ be two gauge-equivalent edge weightings.
    For an edge $e$, let $P_e$ and $P_e'$ be the respective probability matrices.
    Then $P_e$ and $P_e'$ are conjugate.
\end {proposition}
\begin {proof}
    Suppose we number the vertices so that $e = (w_1, b_1)$, and we write the Kasteleyn matrix in block form as
    \[ K = \begin{pmatrix} A&B \\ C & D \end{pmatrix}\]
    where $A = \mathrm{wt}(e)$ is the $n \times n$ matrix on the edge $e$. Then as mentioned above, a gauge transformation by $M$ at the black vertex $b_1$
    changes the Kasteleyn matrix to
    \[ K' = \begin{pmatrix} AM & B \\ CM & D \end{pmatrix} \]
    Since $P_e = (I - A^{-1}BD^{-1}C)^{-1}$, the new matrix after gauge will be 
    \[ P_e' = (I - M^{-1}A^{-1}BD^{-1}CM)^{-1} = M^{-1} P_e M .\]
    A gauge transformation at a white vertex is even simpler, since the contributions of $M$ and $M^{-1}$ cancel, and in that case $P_e = P_e'$.
\end {proof}

We will see in Section~\ref{sec:formulas} that local statistics involving edge multiplicities are expressed in terms of spectral invariants
(i.e. conjugation-invariant quantities) of the matrix $P_e$. As expected, this agrees with Proposition \ref{prop:gauge}, since
the probability measure does not change under gauge.

\begin {example}
Let us return now to Example \ref{ex:dimerwt}, and compute $\mathrm{wt}(\omega)$ in more detail.
We may apply gauge transformations so that all matrices except for $A$ and $D$ are the identity. In this case, note that any coloring that contributes a nonzero term to the sum must have its half-edge labels agree along the edges with weight $I$, which forces the half-edge labels to agree along the edge with weight $A$ as well. It follows that $\mathrm{wt}(\omega ) = \det(D) \cdot \tr(A)$, where the three terms of the trace correspond to the three colorings shown below.

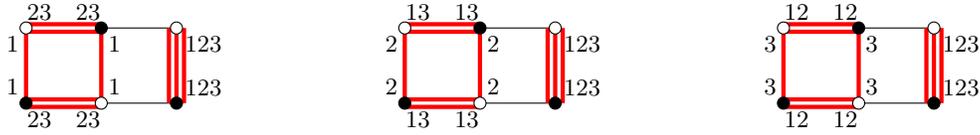
\phantom{=}

\begin{center}
\begin{multicols}{3}
     \begin {tikzpicture}
        \draw (1,1) -- (2,1);
        \draw (1,0) -- (2,0);
        \draw[red, ultra thick] (0,0) -- (0,1);
        \draw[red, ultra thick] (1,0) -- (1,1);
        \draw[red, ultra thick] (0,-0.05) -- (1,-0.05);
        \draw[red, ultra thick] (0,.05) -- (1,0.05);
        \draw[red, ultra thick] (0,0.95) -- (1,0.95);
        \draw[red, ultra thick] (0,1.05) -- (1,1.05);
        \draw[red, ultra thick] (1.9,0) -- (1.9,1);
        \draw[red, ultra thick] (2,0) -- (2,1);
        \draw[red, ultra thick] (2.1,0) -- (2.1,1);
        
        \draw[fill=black] (0,0) circle (0.08) node[below,xshift=5pt]{23} node[above,xshift=-5pt]{1};
        \draw[fill=black] (1,1) circle (0.08) node[below,xshift=5pt]{1} node[above,xshift=-5pt]{23};
        \draw[fill=black] (2,0) circle (0.08) node[right,yshift=6pt]{123};
        \draw[fill=white] (1,0) circle (0.08) node[below,xshift=-5pt]{23} node[above,xshift=5pt]{1};
        \draw[fill=white] (0,1) circle (0.08) node[below,xshift=-5pt]{1} node[above,xshift=5pt]{23};
        \draw[fill=white] (2,1) circle (0.08) node[right,yshift=-6pt]{123};
    \end {tikzpicture}
    
    \columnbreak
     \begin {tikzpicture}
        \draw (1,1) -- (2,1);
        \draw (1,0) -- (2,0);
        \draw[red, ultra thick] (0,0) -- (0,1);
        \draw[red, ultra thick] (1,0) -- (1,1);
        \draw[red, ultra thick] (0,-0.05) -- (1,-0.05);
        \draw[red, ultra thick] (0,.05) -- (1,0.05);
        \draw[red, ultra thick] (0,0.95) -- (1,0.95);
        \draw[red, ultra thick] (0,1.05) -- (1,1.05);
        \draw[red, ultra thick] (1.9,0) -- (1.9,1);
        \draw[red, ultra thick] (2,0) -- (2,1);
        \draw[red, ultra thick] (2.1,0) -- (2.1,1);
        
        \draw[fill=black] (0,0) circle (0.08) node[below,xshift=5pt]{13} node[above,xshift=-5pt]{2};
        \draw[fill=black] (1,1) circle (0.08) node[below,xshift=5pt]{2} node[above,xshift=-5pt]{13};
        \draw[fill=black] (2,0) circle (0.08) node[right,yshift=6pt]{123};
        \draw[fill=white] (1,0) circle (0.08) node[below,xshift=-5pt]{13} node[above,xshift=5pt]{2};
        \draw[fill=white] (0,1) circle (0.08) node[below,xshift=-5pt]{2} node[above,xshift=5pt]{13};
        \draw[fill=white] (2,1) circle (0.08) node[right,yshift=-6pt]{123};
    \end {tikzpicture}
    
    \columnbreak
     \begin {tikzpicture}
        \draw (1,1) -- (2,1);
        \draw (1,0) -- (2,0);
        \draw[red, ultra thick] (0,0) -- (0,1);
        \draw[red, ultra thick] (1,0) -- (1,1);
        \draw[red, ultra thick] (0,-0.05) -- (1,-0.05);
        \draw[red, ultra thick] (0,.05) -- (1,0.05);
        \draw[red, ultra thick] (0,0.95) -- (1,0.95);
        \draw[red, ultra thick] (0,1.05) -- (1,1.05);
        \draw[red, ultra thick] (1.9,0) -- (1.9,1);
        \draw[red, ultra thick] (2,0) -- (2,1);
        \draw[red, ultra thick] (2.1,0) -- (2.1,1);
        
        \draw[fill=black] (0,0) circle (0.08) node[below,xshift=5pt]{12} node[above,xshift=-5pt]{3};
        \draw[fill=black] (1,1) circle (0.08) node[below,xshift=5pt]{3} node[above,xshift=-5pt]{12};
        \draw[fill=black] (2,0) circle (0.08) node[right,yshift=6pt]{123};
        \draw[fill=white] (1,0) circle (0.08) node[below,xshift=-5pt]{12} node[above,xshift=5pt]{3};
        \draw[fill=white] (0,1) circle (0.08) node[below,xshift=-5pt]{3} node[above,xshift=5pt]{12};
        \draw[fill=white] (2,1) circle (0.08) node[right,yshift=-6pt]{123};
    \end {tikzpicture}
\end{multicols}
\end{center}
\end {example}

\section{Local Moves}
\label{sec:localMoves}

In this section, we define \emph{local moves} on our graphs. These local moves, shown in Table~\ref{table:localmove}, each alter a small portion 
of the graph by adding or removing vertices and edges and then appropriately altering the local edge weights. 
Each move is invertible (i.e., can be performed in either direction), although it is typically only useful to perform local moves 
that simplify the overall graph. 
Sequences of local moves can be used to make the application of our formulas more tractable. A small example of this is shown in Remark \ref{remk:snake}.

These local moves are all well-known in the case of the ordinary ($n=1$) dimer model (see e.g. \cite{gk_13}). 
Versions of these moves also appeared in \cite{postnikov} to study network parameterizations of the positive Grassmannian. 
The type (iv) move, which we have called the ``square move'', is also called by many other names,
such as ``urban renewal'' \cite{propp_03}, ``$2 \times 2$ move,'' and ``spider move'' \cite{gk_13}. It appears in the theory of cluster algebras as a manifestation of \emph{mutation}.

\begin{table}
\label{table:localmove}
\begin{tblr}{colspec = {|X[1,c]|X[2,c,h]|}} \hline
     Local Move & \\ \hline \hline 
    (i) Leaf Trimming &
    \begin{tikzpicture}[scale=1.25]
    \draw[dashed,gray] (0,0) circle[radius=30pt];
    \draw (0,0) to (-0.75,0);
    \draw (0,0) to (0.75,0.75);
    \draw (0,0) to (0.75,-0.75);
    \draw (0.5,0.5) to (1,0.4);
    \draw (0.5,0.5) to (0.4,0.955);
    \draw (0.5,-0.5) to (1,-0.4);
    \draw (0.5,-0.5) to (0.4,-0.955);
    \filldraw (0,0) circle[radius=2pt];
    \draw[fill=white, draw=black, line width=1pt] (-0.75,0) circle (2pt);
    \draw[fill=white, draw=black, line width=1pt] (0.5,0.5) circle (2pt);
    \draw[fill=white, draw=black, line width=1pt] (0.5,-0.5) circle (2pt);

    \draw[<->] (2,0) to (3,0);

    \draw[dashed,gray] (5,0) circle[radius=30pt];
    \draw (5.5,0.5) to (6,0.4);
    \draw (5.5,0.5) to (5.4,0.955);
    \draw (5.5,-0.5) to (6,-0.4);
    \draw (5.5,-0.5) to (5.4,-0.955);
    \draw (5.5,0.5) to (5.75,0.75);
    \draw (5.5,-0.5) to (5.75,-0.75);
    \draw[fill=white, draw=black, line width=1pt] (5.5,0.5) circle (2pt);
    \draw[fill=white, draw=black, line width=1pt] (5.5,-0.5) circle (2pt);
    \end{tikzpicture} \\ \hline
    (ii) Parallel Edge Reduction &
    \begin{tikzpicture}[scale=1.25]
    \draw[out=135,in=-135,looseness=1.25] (0,-0.6) to node[left,scale=0.75]{$A_1$} (0,0.6);
    \draw[out=45,in=-45,looseness=1.25] (0,-0.6) to node[right,scale=0.75]{$A_k$} (0.,0.6);
    \node[xshift=1] at (0,0) {$\dots$};
    \draw (0,0.6) to (0.4,0.975);
    \draw (0,0.6) to (-0.4,0.975);
    \draw (0,-0.6) to (0.4,-0.975);
    \draw (0,-0.6) to (-0.4,-0.975);
    \draw[dashed,gray] (0,0) circle[radius=30pt];
    \filldraw (0,0.6) circle[radius=2pt];
    \draw[fill=white, draw=black, line width=1pt] (0,-0.6) circle (2pt);

    \draw[<->] (2,0) to (3,0);

    \draw[dashed,gray] (5,0) circle[radius=30pt];
    \draw (5,0.6) to (5.4,0.975);
    \draw (5,0.6) to (4.6,0.975);
    \draw (5,-0.6) to (5.4,-0.975);
    \draw (5,-0.6) to (4.6,-0.975);
    \draw[out=135,in=-135] (5,-0.6) to node[right,scale=0.65]{$A_1 + \cdots + A_k$} (5,0.6);
    \draw[dashed,gray] (5,0) circle[radius=30pt];
    \filldraw (5,0.6) circle[radius=2pt];
    \draw[fill=white, draw=black, line width=1pt] (5,-0.6) circle (2pt);
\end{tikzpicture} \\ \hline
    (iii) Contraction &
    \begin{tikzpicture}[scale=1.25]
    \draw[dashed,gray] (0,0) circle[radius=30pt];
    \draw (-0.7,0) to (0,0);
    \draw (0,0) to (0.7,0);
    \draw (0.7,0) to node[left,scale=0.75,xshift=4,yshift=10]{$X_1$} (0.95,0.5);
    \draw (0.7,0) to node[left,scale=0.75,xshift=4,yshift=-8]{$X_j$} (0.95,-0.5);
    \node at (0.9,0.1) {$\vdots$};
    \node at (-0.9,0.1) {$\vdots$};
    \draw (-0.7,0) to node[right,scale=0.75,xshift=-4,yshift=10]{$Y_1$} (-0.95,0.5);
    \draw (-0.7,0) to node[right,scale=0.75,xshift=-4,yshift=-8]{$Y_k$} (-0.95,-0.5);
    \draw[fill=white, draw=black, line width=1pt] (0,0) circle (2pt);
    \filldraw (0.7,0) circle[radius=2pt];
    \filldraw (-0.7,0) circle[radius=2pt];

    \draw[<->] (2,0) to (3,0);

    \draw[dashed,gray] (5,0) circle[radius=30pt];
    \draw (5,0) to node[above,scale=0.65,xshift=2,yshift=5]{$X_1$} (5.95,0.5);
    \draw (5,0) to node[above,scale=0.65,xshift=9,yshift=-25]{$X_j$} (5.95,-0.5);
    \node at (5.9,0.1) {$\vdots$};
    \node at (4.1,0.1) {$\vdots$};
    \draw (5,0) to node[above,scale=0.65,xshift=7,yshift=5]{$Y_1$} (4.05,0.5);
    \draw (5,0) to node[below,scale=0.65,xshift=10,yshift=-2]{$Y_k$}(4.05,-0.5);
    \filldraw (5,0) circle[radius=2pt];
    \end{tikzpicture} \\ \hline
    (iv) Square Move &
    \begin{tikzpicture}[scale=1.25]
    \begin {scope}[shift={(5,0)}]
    \draw[dashed,gray] (0,0) circle[radius=30pt];
    \draw (-0.3,-0.3) to node[below,scale=0.65]{$B$} (0.3,-0.3) to node[right,scale=0.65]{$C$} (0.3,0.3) to node[above,scale=0.65]{$D$} (-0.3,0.3) to node[left,scale=0.65]{$A$} (-0.3,-0.3);
    \draw (0.3,0.3) to (0.6,0.6);
    \draw (0.3,-0.3) to (0.6,-0.6);
    \draw (-0.3,0.3) to (-0.6,0.6);
    \draw (-0.3,-0.3) to (-0.6,-0.6);
    \draw (0.6,0.6) to (0.65,0.85);
    \draw (0.6,0.6) to (0.85,0.65);
    \draw (0.6,-0.6) to (0.65,-0.85);
    \draw (0.6,-0.6) to (0.85,-0.65);
    \draw (-0.6,0.6) to (-0.65,0.85);
    \draw (-0.6,0.6) to (-0.85,0.65);
    \draw (-0.6,-0.6) to (-0.65,-0.85);
    \draw (-0.6,-0.6) to (-0.85,-0.65);
    \draw[fill=white, draw=black, line width=1pt] (0.3,0.3) circle (2pt);
    \draw[fill=white, draw=black, line width=1pt] (-0.3,-0.3) circle (2pt);
    \filldraw (0.3,-0.3) circle[radius=2pt];
    \filldraw (-0.3,0.3) circle[radius=2pt];
    \filldraw (0.6,0.6) circle[radius=2pt];
    \filldraw (-0.6,-0.6) circle[radius=2pt];
    \draw[fill=white, draw=black, line width=1pt] (0.6,-0.6) circle (2pt);
    \draw[fill=white, draw=black, line width=1pt] (-0.6,0.6) circle (2pt);
    \end {scope}

    \draw[<->] (2,0) to (3,0);

    \begin {scope}[shift={(-5,0)}]
    \draw[dashed,gray] (5,0) circle[radius=30pt];
    \draw (5.6,0.6) to node[above,scale=0.75]{$d$} (4.4,0.6) to node[left,scale=0.75]{$a$} (4.4,-0.6) to node[below,scale=0.75]{$b$} (5.6,-0.6) to node[right,scale=0.75]{$c$} (5.6,0.6);
    \draw (5.6,0.6) to (5.65,0.85);
    \draw (5.6,0.6) to (5.85,0.65);
    \draw (5.6,-0.6) to (5.65,-0.85);
    \draw (5.6,-0.6) to (5.85,-0.65);
    \draw (4.4,0.6) to (4.35,0.85);
    \draw (4.4,0.6) to (4.15,0.65);
    \draw (4.4,-0.6) to (4.35,-0.85);
    \draw (4.4,-0.6) to (4.15,-0.65);
    \filldraw (5.6,0.6) circle[radius=2pt];
    \filldraw (4.4,-0.6) circle[radius=2pt];
    \draw[fill=white, draw=black, line width=1pt] (5.6,-0.6) circle (2pt);
    \draw[fill=white, draw=black, line width=1pt] (4.4,0.6) circle (2pt);
    \end {scope}

    \node at (5,-1.5) {$A = (a + dc^{-1}b)^{-1}$};
    \node at (5,-2) {$B = (b + cd^{-1}a)^{-1}$};
    \node at (5,-2.5) {$C = (c + ba^{-1}d)^{-1}$};
    \node at (5,-3) {$D = (d + ab^{-1}c)^{-1}$};
    \end{tikzpicture} \\ \hline
\end{tblr}
\caption{A classification of the local moves. The diagrams show the portion of the graph that is transformed; the remainder of the graph, which lies outside of the dashed circle, is unaltered. All unlabeled edges within the circle are assumed to be weighted by the identity matrix.}
\end{table}

\begin{proposition}
    \label{prop:partitionFunctions_localMoves}
    Let $G$ be a graph and $G'$ be a graph obtained from $G$ by a sequence of local moves. Then $Z(G')$ is a scalar multiple of $Z(G)$. In particular:
    \begin{itemize}
        \item If $G'$ is obtained from $G$ by a single local move of types (i), (ii), or (iii), then $Z(G') = Z(G)$.
        \item If $G'$ is obtained from $G$ by a single local move of type (iv), then $Z(G')$ and $Z(G)$ differ by the scalar factor $\textrm{det}\begin{pmatrix} A & B \\ -D & C \end{pmatrix}$, where $A, B, C,$ and $D$ are the new edge weights specified in Table~\ref{table:localmove}.
    \end{itemize}
\end{proposition}

\begin{proof}
    Let $K$ denote the Kastelyn matrix of $G$ and $K'$ denote the Kastelyn matrix of $G'$. We proceed by sequentially considering each type of local move shown in Table~\ref{table:localmove}.

    Suppose that $G'$ is obtained from $G$ by a single local move of type (i). Suppose the vertices are numbered so that $w_1$ and $b_1$ are the pictured white and black vertices, and that the edge $(w_1,b_1)$ is weighted by the identity matrix (using gauge transformations if necessary). Note that $K$ has the block form $\begin{pmatrix} I & \ast \\ 0 & K' \end{pmatrix}$.
    Because it is block triangular, we have $\textrm{det}(K) =\mathrm{det}(K')$, so $Z(G) = Z(G')$.

    Next, suppose that $G'$ is obtained from $G$ by a single local move of type (ii), where the parallel edges in $G$ have endpoints $w_i$ and $b_j$ and weights $A_1, \dots, A_k$. From the definition of the Kastelyn matrix, we know that $K_{[i],[j]} = A_1 + \cdots + A_k$. When the parallel edges in $G$ are replaced by a single edge in $G'$ with weight $A_1 + \cdots + A_k$, we still have $K'_{[i],[j]} = A_1 + \cdots A_k$. Hence, $Z(G') = Z(G)$.

    Now, suppose that $G'$ is obtained from $G$ by a single local move of type (iii). Suppose the vertices are numbered so that $w_1$ is the central white
    vertex, $w_2,\dots,w_{j+1}$ are the other endpoints of the $X$-edges, and $w_{j+2},\dots,w_{j+k-1}$ are the other endpoints of the $Y$-edges (one can do
    a similar analysis if some of these vertices coincide). Also, let $b_1$ and $b_2$ be the two pictured black vertices. Then the Kasteleyn matrix has the form
    \[
        K = \begin{pmatrix}
                I      & I      & 0    & \cdots & 0 \\
                X_1    & 0      & \ast & \cdots & \ast \\
                \vdots & \vdots & \vdots & \ddots & \vdots \\
                X_j    & 0      & \ast   & \cdots & \ast \\
                0      & Y_1    & \ast   & \cdots & \ast \\
                \vdots & \vdots & \vdots & \ddots & \vdots \\
                0      & Y_k    & \ast   & \cdots & \ast
            \end {pmatrix}.
    \]
    Performing a Schur reduction on the upper-left $n \times n$ block, we see that $\det(K)$ is equal to the determinant of the $(|W|-1)n \times (|B|-1)n$ matrix
    \[ 
        \begin{pmatrix}
            0      & \ast \cdots \ast\\
            \vdots & \ast \cdots \ast\\
            0      & \ast \cdots \ast\\
            Y_1    & \ast \cdots \ast\\
            \vdots & \ast \cdots \ast\\
            Y_k    & \ast \cdots \ast\\
            0      & \ast \cdots \ast\\
            \vdots & \ast \cdots \ast\\
            0      & \ast \cdots \ast
        \end{pmatrix}
        -
        \begin{pmatrix} X_1 \\ \vdots \\ X_j \\ 0 \\ \vdots \\ 0 \end{pmatrix} 
        \begin{pmatrix} I & 0 & \cdots & 0 \end{pmatrix}
        =
        \begin{pmatrix}
            -X_1      & \ast \cdots \ast\\
            \vdots & \ast \cdots \ast\\
            -X_j      & \ast \cdots \ast\\
            Y_1    & \ast \cdots \ast\\
            \vdots & \ast \cdots \ast\\
            Y_k    & \ast \cdots \ast\\
            0      & \ast \cdots \ast\\
            \vdots & \ast \cdots \ast\\
            0      & \ast \cdots \ast
        \end{pmatrix}
    \]
    One can see by inspection that this is the Kasteleyn matrix of $G'$. The extra signs on the $X_i$'s can be explained as follows.
    The top and bottom faces in the figure both decrease in number of edges by 2, and so their Kasteleyn signs must be changed. One way to do this is to
    negate the signs on all the $X_i$ edges. Indeed, the top and bottom face are adjacent to one $X$ edge, while the faces on the right are bordered by
    two $X$-edges each, and so their signs will be unchanged.

    Finally, suppose that $G'$ is obtained from $G$ by a single local move of type (iv). Suppose the vertices are numbered so that the pictured vertices come last. If $X$ is the portion of the Kasteleyn matrix corresponding to only the unpictured vertices, then we have
    \[
        K = \begin{pmatrix} X & \ast & \ast \\ \ast & a & -d \\ \ast & b & c \end{pmatrix} \quad \text{and} \quad 
        K' = \left( \begin{array}{ccc|cc}  
                 X & \ast & \ast & 0 & 0 \\ 
                 \ast & 0 & 0 & I & 0 \\
                 \ast & 0 & 0 & 0 & I \\ \hline
                 0 & -I & 0 & A & B \\
                 0 & 0 & -I & -D & C
             \end{array} \right).
    \]
    Observe that the new edge weights $A, B, C,$ and $D$ are chosen precisely so that
    \[ \begin{pmatrix} A & B \\ -D & C \end{pmatrix}^{-1} = \begin{pmatrix} a & -d \\ b & c \end{pmatrix}, \]
    and so $K$ is obtained from $K'$ by Schur reduction on the bottom-right block. Thus, $\det(K)$ and $\det(K')$ differ only by
    the scalar factor $\det \begin{pmatrix} A & B \\ -D & C \end{pmatrix}$.
\end{proof}

\begin{proposition}
    \label{prop:ProbabilityMatrices_localMoves}
    Let $G$ be a graph and $G'$ be a graph obtained from $G$ by a local move. For an edge $e = (w_i, b_j) \in E(G)$ which is not impacted by the local move on $G$ (i.e. at least one of $w_i$ or $b_j$ does not appear in the local move picture), we have that $P_e = P_e'$, where $P_e'$ is the corresponding probability matrix at $e \in E(G')$.
\end{proposition}

\begin{proof}
Let $K$ be the Kasteleyn matrix of $G$ and $K'$ the Kasteleyn matrix of $G'$. Then $P_e = K^{[j],[i]}K_{[i],[j]}$. Let $e' = (w_{i'},b_{j'})$ denote the image of $e$ under the local move on $G$, where in particular the indexing of vertices may change in $G'$. Since $e$ is not impacted by the local move on $G$, note that $\mathrm{wt}(e) = \mathrm{wt}(e')$. 

The proofs for type (i) and (ii) moves are fairly immediate. For a type (i) move, we noted in the proof of Proposition~\ref{prop:partitionFunctions_localMoves} that 
$K = \begin{pmatrix} I & \ast \\ 0 & K' \end{pmatrix}$, and so $K^{-1} = \begin{pmatrix} I & \ast \\ 0 & (K')^{-1} \end{pmatrix}$. 
Hence, $K^{[j],[i]} = (K')^{[j'],[i']}$, and $P_e = P_e'$. A type (ii) move does not change the Kasteleyn matrix ($K = K'$), so clearly $P_e = P_e'$.

For a type (iii) move, recall the forms of $K$ and $K'$ from the proof of Proposition~\ref{prop:partitionFunctions_localMoves}.
We can multiply $K$ on the right by a block upper-triangular matrix to obtain a matrix containing $K'$:
\[ \hat{K} := K \begin{pmatrix} I & -I & 0 \\ 0 & I & 0 \\ 0 & 0 & I \end{pmatrix} = \begin{pmatrix} I & 0 \\ \ast & K' \end{pmatrix}. \]
Just as in the argument above for type (i), $\hat{K}^{-1}$ contains $(K')^{-1}$ as a block (since it is block lower-triangular). Also, 
\[ \hat{K}^{-1} = \begin{pmatrix} I & I & 0 \\ 0 & I & 0 \\ 0 & 0 & I \end{pmatrix} K^{-1}, \] 
and so $\hat{K}^{-1}$ differs from $K^{-1}$ only in the first block row.
If the edge $e$ is not impacted by the type (iii) move, then it corresponds to an entry in $K^{-1}$ which is \emph{not} in the top block row,
and hence the corresponding entries of $K^{-1}$ and $\hat{K}^{-1}$ are the same, and the relevant entry of $\hat{K}^{-1}$ is equal to
the corresponding entry of $(K')^{-1}$ by the argument used for type (i) moves.

Finally, consider the type (iv) move. For a block matrix $M = \begin{pmatrix} A&B\\C&D \end{pmatrix}$, let $\bar{A} := A - BD^{-1}C$ denote the Schur complement of $D$. Then the inverse of $M$ has the form $M^{-1} = \begin{pmatrix} \bar{A}^{-1} & \ast \\ \ast & \ast \end{pmatrix}$. As mentioned in the proof of Proposition~\ref{prop:partitionFunctions_localMoves},
$K$ is the Schur complement of $K'$ by the bottom-right block, and so the upper-left portion of $(K')^{-1}$ (which corresponds to edges not in the local picture) 
is equal to $K^{-1}$.
\end{proof}

\section{Formulas for Local Edge Statistics} \label{sec:formulas}

In this section we will prove several formulas for local statistics on edges for the $M_n$-dimer model on $G$. For each $e$, we write $m_e$
for the multiplicity of edge $e$ in an $n$-dimer cover.
The distributions of these random variables will be expressed in terms
of the probability matrices. To derive these formulas, it will be useful to derive an expression for the probability generating function for edge multiplicities.

For each edge $e$, let $t_e$ be a formal variable. We define a modified version $\widetilde{K}$ of the Kasteleyn matrix (see section \ref{sec:Mn_model})
with entries in the polynomial ring $\R[t_e ~|~ e \in E]$,
where for each edge $e = (w,b)$, we have $\widetilde{K}_{[w],[b]} = t_e K_{[w],[b]}$. An $n$-dimer cover $\omega \in \Omega_n(G)$ 
is determined by its edge multiplicities, and we define an associated monomial $t^\omega := \prod_{e \in E} t_e^{m_e}$.
If we choose a linear ordering of the edges $e_1,e_2,\dots,e_N$, and let $t_i := t_{e_i}$, then the joint
probability generating function for the random variables $m_i := m_{e_i}$ is simply
\[ 
    \frac{\det(\widetilde{K})}{\det(K)} = \sum_{i_1,\dots,i_N} \mathrm{Pr}[m_j=i_j, \, \forall j] \, t_1^{i_1} \cdots t_N^{i_N} 
    = \sum_{\omega \in \Omega_n(G)} \mathrm{Pr}[\omega] \, t^\omega .
\]

To see this, recall the expression for $\mathrm{wt}(\omega)$ from Definition \ref{def:dimerwt}. Each edge with multiplicity $k$
contributes a factor which is a $k \times k$ minor of the matrix on that edge. Thus, if we scale each edge $e$ by $t_e$, then this edge
will give an additional factor of $t_e^k$.
If we are only interested in one edge (say $e_1$), we could set $t_2=t_3=\cdots=t_N = 1$ and $t_1 = t$ to get
the probability generating function for a single edge:
\[ \left. \frac{\det(\widetilde{K})}{\det(K)} \right|_{t_2=\cdots=t_N=1} = \sum_{k} \mathrm{Pr}[m_1=k] \, t^k. \]
Similarly, for any subset of edges, we can set the complementary variables equal to 1 to obtain
the joint probability function for the desired edge multiplicities.

In the single dimer model (when $n=1$), if $p_e$ is the probability of edge $e$ being in a dimer cover,
then the probability generating function for $m_e$ (which can only take values $0$ and $1$) is simply $(1-p_e) + p_e t$.
The following is a generalization of this fact.

\begin {theorem}\label{thm:pgf}
    Let $e$ be an edge of $G$, and $m_e$ be its multiplicity. The probability generating function for the random variable $m_e$ on $\Omega_n(G)$ is
    \[ \sum_{k=0}^n \mathrm{Pr}[m_e=k] \, t^k = \det(I + (t-1)P_e) = \det((I - P_e) + t P_e). \]
\end {theorem}
\begin {proof}
    Note that $\widetilde{K}$ depends linearly on $t$, so we may write $\widetilde{K} = M + tX$ where $M$ and $X$ are constant matrices.
    In particular, assuming the edges are numbered so that $e = (w_1,b_1)$, $X_{[1],[1]} = \mathrm{wt}(e)$ and otherwise $X_{[i],[j]} = 0$. 
    We have the following straightforward calculation:
    \begin {align*}
        \det(\widetilde{K}) &= \det(M+tX) \\
        &= \det(M+X-X+tX) \\
        &= \det(K + (t-1)X) \\
        &= \det(K) \det(I + (t-1)K^{-1}X).
    \end {align*}
    Note that $(K^{-1}X)_{[i],[j]} = 0$ unless $j=1$, and in particular $(K^{-1}X)_{[1],[1]}$ is by definition $P_e$. So $I+(t-1)K^{-1}X$ is a block matrix of the form
    \[ 
        I+(t-1)K^{-1}X = \begin{pmatrix} 
            I + (t-1)P_e & 0 & 0 & \cdots & 0 \\ 
            \ast         & I & 0 & \cdots & 0 \\
            \ast         & 0 & I & \ddots & 0 \\
            \vdots       & \ddots & \ddots & \ddots & \vdots \\
            \ast         & 0      & \cdots & 0 & I
        \end{pmatrix} .
    \]
    In particular, $\det(I + (t-1)K^{-1}X) = \det(I + (t-1)P_e)$.
\end {proof}

\begin {remark} \label{rmk:Poisson_binomial}
    Let $p_1,\dots,p_n$ be the eigenvalues of the matrix $P_e$. Then by Theorem \ref{thm:pgf},
    the probability generating function for the edge multiplicity $m_e$ factors as
    \[ \sum_{k=0}^n \mathrm{Pr}[m_e=k]\, t^k = \prod_{i=1}^n (1 + (t-1)p_i) = \prod_{i=1}^n \left( (1-p_i) + p_i t \right). \]
\end {remark}

\begin {example}
    If all the eigenvalues $p_i$ are in the interval $[0,1]$, then each factor $((1-p_i) + p_i t)$ is the probability generating function
    for a Bernoulli random variable $X_i$ with probability $p_i$ of success, and hence the edge multiplicity $m_e$ has the same distribution
    as a sum $X_1 + X_2 + \cdots + X_n$ of independent Bernoulli variables. Such a distribution is called a \emph{Poisson binomial distribution}.
\end {example}


The formula for the probability generating function in Theorem \ref{thm:pgf} allows us to easily compute various statistical quantities for
the edge multiplicities, which we give below.

\begin {corollary}
    Let $e$ be an edge of $G$, and let $p_1,\dots,p_n$ be the eigenvalues of the probability matrix $P_e$.
    The probability that the edge $e$ appears with a given multiplicity is given by the following expressions.
    Below, $e_k(x_1,\dots,x_n)$ denote the elementary symmetric polynomials.
    \begin {enumerate}
        \item[(a)] $\mathrm{Pr}[m_e=k] = (1-p_1) \cdots (1-p_n) \, e_k \left(\frac{p_1}{1-p_1},\dots,\frac{p_n}{1-p_n} \right)$
        \item[(b)] $\mathrm{Pr}[m_e=k] = p_1 \cdots p_n \, e_{n-k} \left( \frac{1-p_1}{p_1}, \dots, \frac{1-p_n}{p_n} \right)$
        \item[(c)] $\mathrm{Pr}[m_e=k] = \sum_{i=k}^n (-1)^{i-k} \binom{i}{k} e_i(p_1,\dots,p_n)$
    \end {enumerate}
\end {corollary}
\begin{proof}
    If $p_i,\dots,p_n$ are the eigenvalues of $P_e$, then 
    \[ \det(I + \lambda P_e) = \prod_i (1+p_i \lambda) = \sum_{k=0}^n e_k(p_1,\dots,p_n) \lambda^k, \]
    where $e_k$ are the elementary symmetric polynomials. By Theorem \ref{thm:pgf}, the probability generating function is obtained by substituting $\lambda = t-1$,
    and the result follows by extracting the coefficient of $t^k$.
\end{proof}

\begin {corollary}
    The probability that the edge $e$ is used in an $n$-dimer cover is 
    \[ \mathrm{Pr}[m_e \neq 0] = 1 - \det(I-P_e). \]
\end {corollary}
\begin {proof}
    Note that $\mathrm{Pr}[m=0]$ is the constant term in the probability generating function, and substituting $t=0$ gives $\det(I-P_e)$.
\end {proof}

\begin{corollary}\label{cor:expected}
The expected multiplicity of the edge $e$ is the trace of the probability matrix. That is, $\mathbb{E}(m_e) = \mathrm{tr}(P_{e})$.
\end{corollary}
\begin{proof}
If $F(t)$ is a probability generating function for some random variable, then its expected value is the derivative of $F$ at $t=1$.
We use this fact, together with Theorem \ref{thm:pgf} and the \emph{Jacobi formula} $\frac{d}{dt} \det(A) = \det(A) \, \mathrm{tr}\left( A^{-1} \frac{d}{dt}A \right)$:
\begin{align*}
    \mathbb{E}(m_e)&=\left.\frac{d}{dt} \det(I+(t-1)P_e) \right|_{t=1}\\
    &=\left. \det(I+(t-1)P_e) \cdot \tr\left((I+(t-1)P_e)^{-1} \cdot P_e\right)\right|_{t=1}\\
    &=\tr(P_e).
\end{align*}
\end{proof}

\begin{proposition}\label{lem:moments}
Let $e$ be an edge of $G$, and let $p_1,\dots,p_n$ be the eigenvalues of $P_e$.
Let $S(N,k)$ denote the Stirling numbers of the second kind. The moments of the random variable $m_e$ are
\[\mathbb{E}[m_e^N] = \sum_{k=1}^N k! \, S(N,k) e_k(p_1,\dots,p_n).\]
\end{proposition}
\begin{proof}
    The moment generating function is obtained from the probability generating function by the substitution $t \mapsto e^t$, so we have
    \[ \sum_{N \geq 0} \Bbb{E}(m_e^N) \frac{t^N}{N!} = \det(I + (e^t-1)P_e). \]
    It is a straightforward exercise to show that if $f(t)$ is a polynomial, then
    \[ \left. \frac{d^N}{dt^N} f(e^t) \right|_{t=0} = \sum_{k=1}^N S(N,k) f^{(k)}(1). \]
    The derivatives of $\det(I+(t-1)P_e)$ at $t=1$ are simply the coefficients of the characteristic polynomial of $P_e$ (up to factors of $k!$),
    and the result follows.
\end{proof}

\begin{corollary}\label{prop:variance}
    The variance of the edge multiplicity $m_e$ is given by
    \[ \mathrm{Var}(m_e) = \tr(P_e) - \tr\left(P_e^2\right) = \tr\left( P_e(I-P_e) \right). \]
\end{corollary}
\begin{proof}
    In the following calculations, we write $e_k(P_e)$ as an abbreviation for $e_k(p_1,\dots,p_n)$.
    Using Proposition \ref{lem:moments} and Corollary \ref{cor:expected}, we compute:
    \[ \mathrm{Var}(m_e) = \mathbb{E}[m_e^2]-\mathbb{E}[m_e]^2 = S(2,1)e_1(P_e) + 2\cdot S(2,2)e_2(P_e) - e_1(P_e)^2. \]
    Since $S(2,1) = S(2,2) = 1$, $e_1(P_e) = \mathrm{tr}(P_e)$, and $2e_2(P_e) = \mathrm{tr}(P_e)^2 - \mathrm{tr}(P_e^2)$,
    the result follows.
\end{proof}

\section {Correlations for Multiple Edges}
\label{sec:multiple_edges}

In this section, we derive formulas for correlations between edges. In other words, for multiple edges $e_1,e_2,\dots,e_k$,
we find an expression for $\Bbb{E}[m_{e_1} m_{e_2} \cdots m_{e_k}]$, the expected product of the edge multiplicities. 
Recall Theorem \ref{thm:kenyon}(b), which says that when $n=1$, the probability of a random dimer cover containing edges
$e_1,\dots,e_k$ is given by a $k \times k$ minor of $K^{-1}$. We will present an analogous formula when $n > 1$.

\medskip

Before stating the result, we first establish some notation, and prove a helpful lemma. Suppose the edges of the graph are given
a total order $e_1$, $e_2$, $\dots$, $e_N$. Given an ordered list of indices $c = (i_1,i_2,\dots,i_k)$,
let $w_j$ and $b_j$ denote the white and black endpoints of edge $e_{i_j}$ (note that $w_1$, $\dots$, $w_k$ do not necessarily 
have to be distinct, and similarly for the $b_i$'s).
We define the probability matrix of the list $c$ to be
\[ P_c := K_{[w_1],[b_1]} K^{[b_1],[w_2]} K_{[w_2],[b_2]} K^{[b_2],[w_3]} \cdots K_{[w_k],[b_k]} K^{[b_k],[w_1]}. \]

\begin {remark}
    When $k=1$, this does not quite agree with our earlier definition of $P_e$. If $k=1$ and $c = (e) = ((b,w))$, then the above definition of
    $P_c$ would give $K_{[w],[b]}K^{[b],[w]}$, whereas the earlier definition had $P_e=K^{[b],[w]}K_{[w],[b]}$. However,
    the distinction is not important, since all relevant formulas involve the trace of $P_e$, and $\mathrm{tr}(AB) = \mathrm{tr}(BA)$.
\end {remark}

\begin {definition} \label{def:app_defn}
    Let $A_1,\dots,A_n$ be square matrices, and $c = (i_1,i_2,\dots,i_k)$ a cyclically ordered set of indices. Define
    \[ \mathrm{Tr}_c(A_1,\dots,A_n) := \mathrm{tr}(A_{i_1} A_{i_2} \cdots A_{i_k}). \]
    For a permutation $\sigma \in S_n$, let $\mathrm{Tr}_\sigma(A_1,\dots,A_n)$ be the product of $\mathrm{Tr}_c$ over all cycles of $\sigma$.
    Finally, let $\Psi_n := \sum_{\sigma \in S_n} \mathrm{sign}(\sigma) \mathrm{Tr}_\sigma$.
\end {definition}

\begin {lemma}\label{lem:derivs}
    Let $A$ be a square matrix whose entries depend on variables $t_1,\dots,t_k$, such that each entry depends on at most one variable.
    That is, $\frac{\partial^2}{\partial t_i \partial t_j} A = 0$ for all $i \neq j$. Then 
    \[ \frac{\partial^k}{\partial t_1 \cdots \partial t_k} \det(A) = \det(A) \, \Psi_k\left( A^{-1} \frac{\partial A}{\partial t_1}, \, \cdots, \, A^{-1} \frac{\partial A}{\partial t_k} \right). \]
\end {lemma}
\begin {proof}
    Induct on $k$. The base cases $(k=1)$ is simply the Jacobi derivative formula: $\frac{d}{dt} \det(A) = \det(A) \mathrm{tr}\left(A^{-1} \frac{d A}{dt} \right)$.

    \medskip

    Let us use the notation $Y_i := A^{-1} \frac{\partial A}{\partial t_i}$.
    For $k > 1$, differentiate the expression for $k-1$ and use the Leibniz rule to get
    \[ 
        \frac{\partial^k}{\partial t_1 \cdots \partial t_k} \det(A)
        = \frac{\partial}{\partial t_k} \det(A) \cdot \Psi_{k-1}(Y_1,\dots,Y_{k-1})
        + \det(A) \cdot \frac{\partial}{\partial t_k} \Psi_{k-1}(Y_1,\dots,Y_{k-1}).
    \]
    For the left term, use the $k=1$ Jacobi formula and then combine:
    \[ 
        \frac{\partial^k}{\partial t_1 \cdots \partial t_k} \det(A) = 
        \det(A) \left[ \mathrm{tr}(Y_k) \Psi_{k-1} + \frac{\partial}{\partial t_k} \Psi_{k-1}(Y_1,\dots,Y_{k-1}) \right].
    \]

    It remains to show that $\Psi_k(Y_1,\dots,Y_k) = \mathrm{tr}(Y_k) \Psi_{k-1}(Y_1,\dots,Y_{k-1}) + \frac{\partial}{\partial t_k} \Psi_{k-1}(Y_1,\dots,Y_{k-1})$.
    Note that permutations $\sigma \in S_k$ fall into two cases: either $k$ is a fixed point of $\sigma$, or not. For $\sigma$ in the first case,
    the cycle notation looks like a permutation in $S_{k-1}$ plus the 1-cycle $(k)$. So the sum of these terms in $\Psi_k$ is given by
    $\mathrm{tr}(Y_k) \Psi_{k-1}(Y_1,\dots,Y_{k-1})$. Lastly, we must check that the remaining terms in $\Psi_k$
    (corresponding to $\sigma \in S_k$ for which $k$ is not a fixed point) are given by $\frac{\partial}{\partial t_k}\Psi_{k-1}(Y_1,\dots,Y_{k-1})$.

    For each $\sigma \in S_{k-1}$, we have a term $\mathrm{sign}(\sigma) \mathrm{Tr}_\sigma$, which is itself a product of factors of the form
    $\mathrm{Tr}_c$, for the cycles $c$ in $\sigma$. By the Leibniz rule, if $\sigma$ has cycles $c_1 c_2 \cdots c_j$, then
    \[ 
        \frac{\partial}{\partial t_k} \mathrm{Tr}_\sigma = 
        \sum_{i=1}^j \mathrm{Tr}_{c_1} \cdots \mathrm{Tr}_{c_{i-1}} \left( \frac{\partial}{\partial t_k} \mathrm{Tr}_{c_i} \right) \mathrm{Tr}_{c_{i+1}} \cdots \mathrm{Tr}_{c_j}.
    \]
    Since second mixed partial derivatives of $A$ vanish, we have $\frac{\partial}{\partial t_k} Y_i = -Y_k Y_i$.
    From this, and the linearity of the trace, we get for each cycle $c = (i_1,i_2,\dots,i_\ell)$:
    \begin {align*}
        \frac{\partial}{\partial t_k} \mathrm{Tr}_c(Y_1,\dots,Y_{k-1}) 
        &= \frac{\partial}{\partial t_k} \mathrm{tr}( Y_{i_1} Y_{i_2} \cdots Y_{i_\ell}) \\
        &= - \sum_{s=1}^\ell \mathrm{tr}(Y_{i_1} \cdots Y_{i_{s-1}} Y_k Y_{i_s} \cdots Y_{i_\ell}).
    \end {align*}
    On the other hand, permutations in $S_k$ where $k$ is not a fixed point come from permutations in $S_{k-1}$ by inserting $k$ into one of the cycles.
    So these expressions correspond exactly to the remaining terms in $\Psi_k$. Note the minus sign in the last line above agrees with the fact that adding
    a $k$ into one of the cycles gives a permutation with opposite sign.
\end {proof}

\begin {theorem}
    Let $e_1,\dots,e_k$ be distinct edges, with $e_i = (w_i,b_i)$. Then 
    \[ 
        \Bbb{E}[m_{e_1} m_{e_2} \cdots m_{e_k}] = 
        \sum_{\sigma \in S_k} \mathrm{sign}(\sigma) \prod_{\substack{\text{cycles } c \\ \text{in } \sigma}} \mathrm{tr}(P_c). 
    \]
\end {theorem}
\begin {proof}
    As mentioned at the start of Section~\ref{sec:formulas}, the joint probability generating function is given by $\det(K^{-1} \widetilde{K}(t_1,\dots,t_k))$,
    where $\widetilde{K}(t_1,\dots,t_k)$ is obtained from $K$ by rescaling the $n \times n$ block corresponding to edge $e_i$ by a factor of $t_i$.
    We now proceed in a manner similar to the proof of Theorem \ref{thm:pgf}.
    
    A total order of the white and black vertices must be chosen to write the matrix $K$.
    Suppose that $i_1,\dots,i_k$ and $j_1,\dots,j_k$ are such that $e_\ell = (w_{i_\ell},b_{j_\ell})$ for $1 \leq \ell \leq k$. 
    Let $B_i$ be a matrix of the same size as $K$,
    such that $(B_\ell)_{[i_\ell],[j_\ell]} = K_{[i_\ell],[j_\ell]}$ and otherwise $B_{[w],[b]} = 0$. In other words, $B_\ell$ agrees with $K$ in the block
    corresponding to the edge $e_\ell$, and is otherwise all zeros. Then
    \[ \widetilde{K}(t_1,\dots,t_k) = K + \sum_{\ell=1}^k (t_\ell-1)B_\ell. \]
    Multiplying by $K^{-1}$, we then have $K^{-1} \widetilde{K}(t_1,\dots,t_k) = I + \sum_j (t_\ell-1)K^{-1} B_\ell$.
    The product $K^{-1}B_\ell$ is a matrix of mostly zeros, except its $j_\ell$-th block
    column is the $i_\ell$-th block column of $K^{-1}$, multiplied on the right by $K_{[i_\ell],[j_\ell]}$ 
    (i.e. consisting of blocks $K^{[1],[i_\ell]} K_{[i_\ell],[j_\ell]}$, $K^{[2],[i_\ell]}K_{[i_\ell],[j_\ell]}$, $\dots$).

    Let $X_\ell := K^{-1}B_\ell$ be these matrices, and $A := I + \sum_\ell (t_\ell-1)X_\ell$. Then, as mentioned above, the joint
    probability generating function for $m_{e_1},\dots,m_{e_k}$ is $\det(A)$. The moment generating function is therefore obtained
    by substituting $t_i \mapsto e^{t_i}$. Let $\widetilde{A} = A(e^{t_1},\dots,e^{t_k})$. Our desired expected value is therefore given by
    \[ \Bbb{E}[m_{e_1} \cdots m_{e_k}] = \left. \frac{\partial^k}{\partial t_1 \cdots \partial t_k} \det(\widetilde{A}) \right|_{t_i = 0, \, \forall i}
        = \left. \frac{\partial^k}{\partial t_1 \cdots \partial t_k} \det \left( I + \sum_\ell (e^{t_\ell}-1) X_\ell \right) \right|_{t_i = 0, \, \forall i}. \]

    Applying Lemma \ref{lem:derivs} to this equation, noting that here $\left. \widetilde{A} \right|_{t_i = 0, \, \forall i} = I$ 
    and $\left. \frac{\partial \widetilde{A}}{\partial t_j} \right|_{t_i = 0, \, \forall i} = X_j$, we arrive at the expression
    \[ \Bbb{E}[m_{e_1} \cdots m_{e_k} ] = \Psi_k(X_1,\dots,X_k). \]
    To prove the theorem, it just remains to see that for any cyclically ordered indices $c = (r_1,r_2,\dots,r_s)$, we have
    $\mathrm{Tr}_c(X_1,\dots,X_k) = \mathrm{tr}(X_{r_1}X_{r_2} \cdots X_{r_s}) = \mathrm{tr}(P_c)$. Each matrix $X_\ell$ is a block matrix, and we have
    \[ 
        \mathrm{tr}(X_{r_1} \cdots X_{r_s}) = 
        \sum_{a_0,a_1,\dots,a_{s-1}} \mathrm{tr}((X_{r_1})_{[a_0],[a_1]} (X_{r_2})_{[a_1],[a_2]} \cdots (X_{r_s})_{[a_{s-1}],[a_0]}).
    \]
    However, since $X_\ell$ is zero except in block-column $j_\ell$, all terms in this sum are zero except one:
    \[ 
        \mathrm{tr}(X_{r_1} \cdots X_{r_s}) = \mathrm{tr} \left( (X_{r_1})_{[j_{r_s}],[j_{r_1}]} (X_{r_2})_{[j_{r_1}],[j_{r_2}]} \cdots (X_{r_s})_{[r_{s-1}],[r_s]} \right) .
    \]
    Comparing this with the structure of the $X_\ell$ matrices mentioned above, we see that this is indeed equal to $\mathrm{tr}(P_c)$.
\end {proof}

\begin{corollary}\label{cor:covariance}
    Let $e_1,e_2$ be two distinct edges, with $e_i = (w_i,b_i)$. The covariance between their edge multiplicities is
    \[\mathrm{Cov}(m_{e_1},m_{e_2}) = -\mathrm{tr}(K_{[w_1],[b_1]} K^{[b_1],[w_2]} K_{[w_2],[b_2]} K^{[b_2],[w_1]}).\]
\end{corollary}
\begin{proof}
    By definition, $\mathrm{Cov}(m_{e_1},m_{e_2}) = \Bbb{E}[m_{e_1}m_{e_2}] - \Bbb{E}[m_{e_1}] \Bbb{E}[m_{e_2}]$. The result
    follows immediately from the previous theorem.
\end{proof}

\section {Examples and Applications}

In this section, we give some explicit applications of our formulas to product measures and the $2 \times N$ grid graph. We then illustrate how the local moves from Section~\ref{sec:localMoves} can be used to reduce an arbitrary snake graph to a $2 \times N$ grid graph, as an example of how these local moves can be used to simplify a graph in order to make explicit application of our formulas more tractable.

\subsection {Product Measures}

Consider the case $\mathrm{wt} \colon E \to \mathrm{Mat}_n(\Bbb{R})$ where all edge weights are diagonal matrices. 
In this case, the probability measure on $\Omega_n(G)$ is 
equal to the product of $n$ probability measures on $\Omega_1(G)$, as follows. For each $i=1,\dots,n$ define an edge weight
function $\omega_i \colon E \to \Bbb{R}$ by declaring that $\omega_i(e) = \mathrm{wt}(e)_{ii}$, the $i$-th diagonal entry of the matrix
on edge $e$. Each $\omega_i$ defines a probability
measure $\mu_i$ on $\Omega_1(G)$, the set of single dimer covers. The measure on $\Omega_n(G)$ given by $\mathrm{wt}$ is 
then the product $\mu_1 \times \cdots \times \mu_n$.

The probability matrices $P_e$ are simply diagonal matrices whose diagonal entries are the edge probabilities with respect to
each $\mu_i$ measure. In this case the conclusion of Remark \ref{rmk:Poisson_binomial} is rather obvious, since the measure itself
is given by sampling $n$ independent single dimer covers. 

\subsection {$2 \times N$ Grid and Noncommutative Continued Fractions}

Let $G_N$ be the $2 \times (N+1)$ grid graph. We call the matrix edge weights $A_i$, $B_i$, $C_i$ as in Figure \ref{fig:2-by-n} ($N=4$ is pictured).

\begin {figure}
\centering
\begin {tikzpicture}[scale=1.2]
    \draw (0,0) grid (4,1);
    
    \draw[blue, line width=1] (0,0) --++ (-135:0.25);
    \draw[blue, line width=1] (1,0) --++ (-90:0.25);
    \draw[blue, line width=1] (2,0) --++ (-90:0.25);
    \draw[blue, line width=1] (3,0) --++ (-90:0.25);
    \draw[blue, line width=1] (4,0) --++ (-45:0.25);
    \draw[blue, line width=1] (4,1) --++ (45:0.25);
    \draw[blue, line width=1] (3,1) --++ (90:0.25);
    \draw[blue, line width=1] (2,1) --++ (90:0.25);
    \draw[blue, line width=1] (1,1) --++ (90:0.25);
    \draw[blue, line width=1] (0,1) --++ (135:0.25);
    
    \foreach \x/\y in {0/0, 1/1, 2/0, 3/1, 4/0} {
        \draw[fill=black] (\x,\y) circle (0.08);
    }
    \foreach \x/\y in {0/1, 1/0, 2/1, 3/0, 4/1} {
        \draw[fill=white] (\x,\y) circle (0.08);
    }

    \draw (0,0.5) node[left] {$B_0$};
    \draw (1,0.5) node[left] {$B_1$};
    \draw (2,0.5) node[left] {$B_2$};
    \draw (3,0.5) node[left] {$B_3$};
    \draw (4,0.5) node[left] {$B_4$};

    \draw (0.5,0) node[below] {$-C_1$};
    \draw (1.5,1) node[above] {$-C_2$};
    \draw (2.5,0) node[below] {$-C_3$};
    \draw (3.5,1) node[above] {$-C_4$};
    
    \draw (0.5,1) node[above] {$A_1$};
    \draw (1.5,0) node[below] {$A_2$};
    \draw (2.5,1) node[above] {$A_3$};
    \draw (3.5,0) node[below] {$A_4$};
\end {tikzpicture}
\caption {The $2 \times 5$ grid graph $G_4$ with arbitrary matrix edge weights (Kasteleyn signs are included). Cilia are drawn as small blue lines.}
\label {fig:2-by-n}
\end {figure}
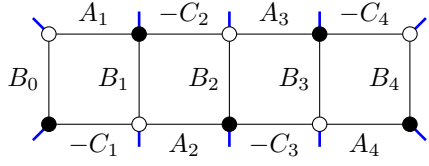

If the white and black vertices are ordered from left-to-right, then the Kasteleyn matrix is a block tridiagonal matrix of the form
\[ 
    K = \begin{pmatrix}
        B_0     & A_1    & 0        & \cdots  & 0 \\
        -C_1    & B_1    & A_2      & \cdots  & 0 \\
        \vdots  & \ddots & \ddots   & \ddots  & \vdots \\
        0       & \cdots & -C_{N-1} & B_{N-1} & A_N \\
        0       & \cdots & 0        & -C_N    & B_N 
    \end{pmatrix}. 
\]
Since the measure on $\Omega_n(G_N)$ is invariant under gauge transformations (see section \ref{sec:pmat}), 
we may assume for simplicity (as long as all edge matrices are invertible)
that all $A_i$ and $C_i$ matrices are equal to the identity.

It is well-known that determinants and inverses of tridiagonal matrices
are related to continued fractions, via \emph{Euler continuants} (see e.g. \cite{mgo_farey}).
Similarly, for block tridiagonal matrices, we will get noncommutative analogs of continued fractions.
For matrices $B_0,B_1,\dots,B_N$, let $[B_0,B_1,\dots,B_N]$ denote the following noncommutative continued fraction:
\[ [B_0,\dots,B_N] := B_0 + (B_1 + (B_2 + (\cdots + B_N^{-1})^{-1} \cdots)^{-1})^{-1}. \]
This expression is a noncommutative version of the continued fraction
\[ b_0 + \cfrac{1}{b_1 + \cfrac{1}{b_2 + \cfrac{1}{\ddots + \cfrac{1}{b_N}}}} \;, \]
and the two expressions coincide when $n=1$ (i.e. when the $B_i$'s are scalars).
More specifically, let $M_N = B_N$, and recursively define $M_i = B_i + M_{i+1}^{-1}$. Then $[B_0,\dots,B_N]$ is defined to be $M_0$.
Such noncommutative continued fractions have been studied before \cite{wedderburn}, 
especially in relation to inverting block tridiagonal matrices \cite{doliwa} \cite{rh_12}. 

For $i<j$, let us denote $F_{i,j} = [B_i,B_{i+1},\dots,B_j]$, and for $i > j$, let $F_{i,j} = [B_i,B_{i-1},\dots,B_j]$.
From formulas which can be found in the aforementioned
references (and are easily derived by Schur reduction on $K$), 
the diagonal blocks of $K^{-1}$ are given by $K^{[i],[i]} = (F_{i,N} + F_{i,0} - B_i)^{-1}$. 
We then have the following.

\medskip

\begin {proposition}
    Let $G_N$ be the $2 \times (N+1)$ grid graph, with vertical edges labeled by $B_0,B_1,\dots,B_N$, and all horizontal edges equal to $I$.
    For the vertical edge $e_i$ (labeled $B_i$), the probability matrix is
    \[ P_{e_i} = \left( B_i^{-1}F_{i,N} + B_i^{-1} F_{i,0} - I \right)^{-1}. \]
\end {proposition}

\medskip

\begin {remark}
    In the two extreme cases (when $i=0$ or $i=N$), either $F_{i,0}$ or $F_{i,N}$ is equal to $B_i$, and two of the three terms cancel. Specifically,
    $K^{[0],[0]} = F_{0,N}^{-1}$ and $K^{[N],[N]} = F_{N,0}^{-1}$. We may therefore interpret the general expression in terms of $P_e$ matrices
    on smaller subgraphs. Let $G' \cong G_{i-1}$ be the subgraph obtained by deleting all vertices to the right of edge $e_i$, and let $G'' \cong G_{N-i}$
    be the subgraph obtained by deleting all vertices to the left of $e_i$. Let $P_{e_i}'$ and $P_{e_i}''$ be the corresponding probability matrices 
    on the $e_i$ edge within $G'$ and $G''$. Then 
    \[ 
        P_{e_i} = ((P'_{e_i})^{-1} + (P''_{e_i})^{-1} - I)^{-1} 
                = P''_{e_i} \left( P'_{e_i} + P''_{e_i} - P'_{e_i} P''_{e_i} \right)^{-1} P'_{e_i}.
    \]
\end {remark}

\medskip

The horizontal edges are all of the form $e_{i,i+1} = (w_i,b_{i+1})$ or $e_{i+1,i} = (w_{i+1},b_i)$. We must therefore examine the blocks of $K^{-1}$
of the form $K^{[i],[i+1]}$ and $K^{[i+1],[i]}$. 

\begin {proposition}
    The $P_e$ matrices for horizontal edges are given by 
    \[ 
        P_{e_{i+1,i}} = -K^{[i],[i+1]} = \left( I + F_{i+1,N} F_{i,0} \right)^{-1} \quad \text{ and } \quad
        P_{e_{i,i+1}} = K^{[i+1],[i]} = \left( I + F_{i,0} F_{i+1,N} \right)^{-1}.
    \]
\end {proposition} 

We can use Corollary \ref{cor:covariance} to compute the covariance between different edges. For that, we will need expressions for more general
blocks of $K^{-1}$. One can compute (see for instance \cite{rh_12}) that 
\[ 
    K^{[i],[j]} = \begin{cases}
        (-1)^{j-i}F_{i,0}^{-1} F_{i+1,0}^{-1} \cdots F_{j-1,0}^{-1} K^{[j],[j]}  & \text{ if  } i<j \\[1ex]
        F_{i,N}^{-1} F_{i-1,N}^{-1} \cdots F_{j+1,N}^{-1} K^{[j],[j]} & \text{ if  } i > j
    \end{cases}
\] 

From this we get the following:

\medskip

\begin {proposition}
    Let $i < j$, and let $e_i$ and $e_j$ be the corresponding vertical edges in $G_N$. Then the covariance between the edge multiplicities is
    \[ 
        \mathrm{Cov}(m_{e_i},m_{e_j}) = (-1)^{j-i}
        \mathrm{tr} \left( P_{e_i} F_{i,0}^{-1} F_{i+1,0}^{-1} \cdots F_{j-1,0}^{-1} P_{e_j} F_{j,N}^{-1} F_{j-1,N}^{-1} \cdots F_{i+1,N}^{-1} \right). 
    \]
\end {proposition}

\medskip

\begin {remark}
    When $n=1$ and all edge weights are 1, this continued fraction is equal to $\frac{f_N}{f_{N+1}}$, the ratio of consecutive Fibonacci numbers.
    More generally, if we choose weights $a_{2i} = q$ and $a_{2i+1} = q^{-1}$, and all $b_i=c_i = 1$, then we obtain $\left[ \frac{f_N}{f_{N+1}} \right]_q$,
    the $q$-rational as defined in \cite{mgo}. More specifically, we can define \emph{$q$-Fibonacci polynomials} $\mathcal{F}_n(q)$ and $\widetilde{\mathcal{F}}(q)$
    by $\mathcal{F}_0(q) = \widetilde{\mathcal{F}}_0(q) = \mathcal{F}_1(q) = \widetilde{\mathcal{F}}_1(q) = 1$, and the recurrences
    \[ 
        \mathcal{F}_n(q) = \begin{cases} 
            q \mathcal{F}_{n-1}(q) + \mathcal{F}_{n-2}(q) & \text{ if $n$ is even} \\
            \mathcal{F}_{n-1}(q) + q^2 \mathcal{F}_{n-2}(q) & \text{ if $n$ is odd}
        \end {cases}, \quad
        \widetilde{\mathcal{F}}_n(q) = \begin{cases} 
            \widetilde{\mathcal{F}}_{n-1}(q) + q^2 \widetilde{\mathcal{F}}_{n-2}(q) & \text{ if $n$ is even} \\
            q \widetilde{\mathcal{F}}_{n-1}(q) + \widetilde{\mathcal{F}}_{n-2}(q) & \text{ if $n$ is odd}
        \end {cases}
    \]
    Then, as mentioned in \cite{mgo}, we will have $\left[ \frac{f_N}{f_{N+1}} \right]_q = \frac{q \widetilde{\mathcal{F}}_N(q)}{\mathcal{F}_{N+1}(q)}$.
    As we observed above, this will also be the probability that the left-most vertical edge is used in a dimer cover. For the $M_n$-dimer model (with $n > 1$)
    if we fix a matrix $Q$, and choose edge weights $B_i = C_i = I$ and $A_{2i} = Q$ and $A_{2i+1} = Q^{-1}$, then the expected multiplicity of the left-most
    vertical edge will be $\mathrm{tr} \left( Q \widetilde{\mathcal{F}}_N(Q) \mathcal{F}_{N+1}(Q)^{-1} \right) = \mathrm{tr} \left( \left[ \frac{f_N}{f_{N+1}} \right]_Q \right)$.
\end {remark}

\begin{remark}\label{remk:snake}
The local moves described in Section~\ref{sec:localMoves} allow us to extend our results for the $2 \times N$ grid graph to arbitrary \emph{snake graphs}. 
Snake graphs (appearing e.g. in \cite{propp_frieze} \cite{msw_13}) are planar bipartite graphs that consist of a sequence of square tiles
such that each consecutive pair of tiles are glued together either along the north edge of the first tile and the south edge of the second tile, 
or along the east edge of the first tile and west edge of the second tile. 
Any snake graph can be transformed into a straight snake graph (i.e., a $2 \times N$ grid graph) through a sequential application of type (iii) contraction moves and type (ii) parallel edge mergers, which allows us to compute statistics for any unaffected edge, e.g. any edge that is not part of a corner tile.
\end{remark}

\begin {example}
Consider the snake graph with matrix edge weights pictured in Figure \ref{fig:snake_moves}.
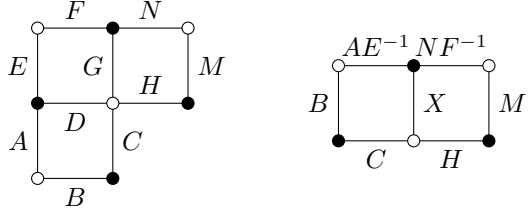
\begin {figure}
    \centering
    \begin {tikzpicture}
        \draw (0,0) -- (1,0) -- (1,1) -- (0,1) -- cycle;
        \draw (2,1) -- (2,2) -- (1,2) -- (1,1) -- cycle;
        \draw (0,1) -- (0,2) -- (1,2);
        
        \draw[fill=white] (0,0) circle (0.08);
        \draw[fill=white] (1,1) circle (0.08);
        \draw[fill=white] (0,2) circle (0.08);
        \draw[fill=black] (1,0) circle (0.08);
        \draw[fill=black] (0,1) circle (0.08);
        \draw[fill=black] (2,1) circle (0.08);
        \draw[fill=white] (2,2) circle (0.08);
        \draw[fill=black] (1,2) circle (0.08);

        \draw (0,0.5) node[left] {$A$};
        \draw (0.5,0) node[below] {$B$};
        \draw (1,0.5) node[right] {$C$};
        \draw (0.5,1) node[below] {$D$};
        \draw (0,1.5) node[left] {$E$};
        \draw (0.5,2) node[above] {$F$};
        \draw (1.5,1) node[above] {$H$};
        \draw (1,1.5) node[left] {$G$};
        \draw (2,1.5) node[right] {$M$};
        \draw (1.5,2) node[above] {$N$};

        \begin {scope}[shift={(4,0.5)}]
            \draw (0,0) -- (1,0) -- (2,0) -- (2,1) -- (1,1) -- (0,1) -- cycle;
            \draw (1,0) -- (1,1);
            
            \draw[fill=black] (0,0) circle (0.08);
            \draw[fill=black] (1,1) circle (0.08);
            \draw[fill=black] (2,0) circle (0.08);
            \draw[fill=white] (1,0) circle (0.08);
            \draw[fill=white] (0,1) circle (0.08);
            \draw[fill=white] (2,1) circle (0.08);
    
            \draw (0,0.5) node[left] {$B$};
            \draw (0.5,0) node[below] {$C$};
            \draw (1,0.5) node[right] {$X$};
            \draw (0.5,1) node[above] {$AE^{-1}$};
            \draw (1.5,0) node[below] {$H$};
            \draw (2,0.5) node[right] {$M$};
            \draw (1.5,1) node[above] {$NF^{-1}$};
        \end {scope}
    \end {tikzpicture}
    \caption{(Left) A snake graph with matrix edge weights. (Right) The result of applying a contraction at the upper-left vertex, followed by a parallel edge reduction.}
    \label{fig:snake_moves}
\end {figure}
After applying gauge transformations, we can contract the edges labeled by $E$ and $F$ and subsequently merge the edges labeled by $D$ and $G$.
Following from Proposition~\ref{prop:ProbabilityMatrices_localMoves} of Section \ref{sec:localMoves}, the probability matrix of any edge that is not impacted by this sequence of local moves (i.e., all edges except those labeled by $D, E, F,$ and $G$) is equivalent to the probability matrix of the corresponding edge in the $2\times 3$ grid graph on the right side of Figure \ref{fig:snake_moves}, where $X \coloneqq DE^{-1}+GF^{-1}$.

\end {example}

\section {An Application to Vertex Models}

\subsection {Mixed Dimer Covers}

In this section, we briefly discuss the extension of our earlier results to the more general setting of \emph{mixed dimer covers}, giving some specific examples. 

\begin {definition}
    Let $\vec{n} = (n_v)_{v \in V} \in \Bbb{N}^V$ be an integer labeling of the vertices of a graph. 
    An \emph{$\vec{n}$-dimer cover} (or a \emph{mixed dimer cover}) is a multiset of edges such that
    each vertex is covered $n_v$ times (counted with multiplicity). Equivalently, it is a function $m \colon E \to \Bbb{N}$
    such that at each vertex $v$, we have $\sum_{e \sim v} m(e) = n_v$.
\end {definition}

To each edge $e = (w,b)$, we assign an $n_w \times n_b$ matrix. The weight of each $\vec{n}$-dimer cover is defined exactly
as in Definition \ref{def:dimerwt}, as a weighted sum over half-edge colorings. 
It was shown in \cite{ko_23} that even in this more general case, the partition
function is still given by the determinant of the appropriate block Kasteleyn matrix.

\begin {example}
    Consider the $2 \times 3$ grid graph pictured in Figure \ref{fig:mixed_ex} with vertex multiplicities $1,2,3$. The edge labels $a,B,C$ are
    $1 \times 1$, $2 \times 2$, and $3 \times 3$ matrices, respectively. There are five mixed dimer configurations. The first four each have
    only a single (non-zero) coloring, making their weights easily apparent. The last configuration has four colorings, pictured at the bottom
    of the figure, and its weight is the sum of four terms corresponding to these colorings.  The Kasteleyn matrix for this graph is
    \[
        K = \left( \begin{array}{c|cc|ccc}
            a  & 1      & 0      & 0      & 0      & 0 \\ \hline
            -1 & b_{11} & b_{12} & 1      & 0      & 0 \\
            0  & b_{21} & b_{22} & 0      & 1      & 0 \\ \hline
            0  & -1     & 0      & c_{11} & c_{12} & c_{13} \\
            0  & 0      & -1     & c_{21} & c_{22} & c_{23} \\
            0  & 0      & 0      & c_{31} & c_{32} & c_{33}
        \end{array} \right).
    \]
    The determinant of this matrix is the sum of the five weights in the figure.

    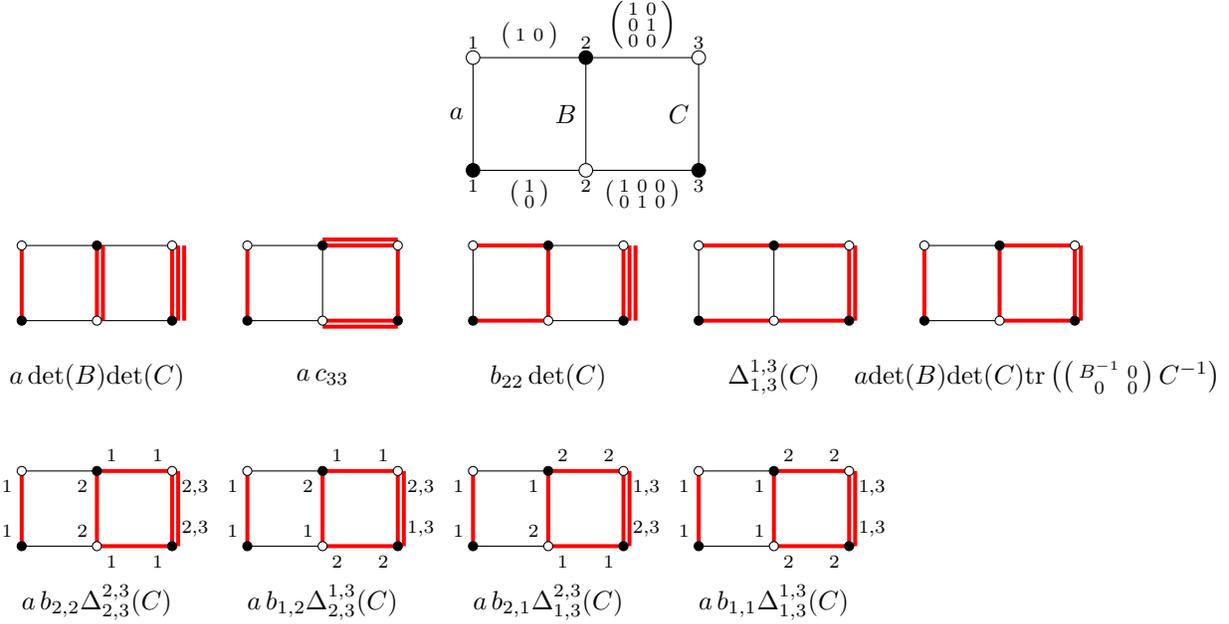
\begin {figure}
    \centering
    \begin {tikzpicture}

        \begin {scope}[shift={(6,2)}, scale=1.5]
            \draw (0,0) grid (2,1);
    
            \draw[fill=black] (0,0) circle (0.06);
            \draw[fill=black] (1,1) circle (0.06);
            \draw[fill=black] (2,0) circle (0.06);
            \draw[fill=white] (0,1) circle (0.06);
            \draw[fill=white] (1,0) circle (0.06);
            \draw[fill=white] (2,1) circle (0.06);
    
            \draw (0,0) node[below] {\scriptsize $1$};
            \draw (0,1) node[above] {\scriptsize $1$};
            \draw (1,0) node[below] {\scriptsize $2$};
            \draw (1,1) node[above] {\scriptsize $2$};
            \draw (2,0) node[below] {\scriptsize $3$};
            \draw (2,1) node[above] {\scriptsize $3$};

            \draw (0,0.5) node[left] {$a$};
            \draw (1,0.5) node[left] {$B$};
            \draw (2,0.5) node[left] {$C$};

            \draw (0.5,1) node[above] {$\left( \begin{smallmatrix} 1 & 0 \end{smallmatrix} \right)$};
            \draw (0.5,0) node[below] {$\left( \begin{smallmatrix} 1 \\ 0 \end{smallmatrix} \right)$};
            \draw (1.5,1) node[above] {$\left( \begin{smallmatrix} 1 & 0 \\ 0 & 1 \\ 0 & 0 \end{smallmatrix} \right)$};
            \draw (1.5,0) node[below] {$\left( \begin{smallmatrix} 1 & 0 & 0 \\ 0 & 1 & 0 \end{smallmatrix} \right)$};
        \end {scope}
        
        \draw (0,0) grid (2,1);

        \draw[red, line width=1.5] (0,0) -- (0,1);
        \draw[red, line width=1.5] (1,0) -- (1,1);
        \draw[red, line width=1.5] (1.08,0) -- (1.08,1);
        \draw[red, line width=1.5] (2,0) -- (2,1);
        \draw[red, line width=1.5] (2.08,0) -- (2.08,1);
        \draw[red, line width=1.5] (2.16,0) -- (2.16,1);
        
        \draw[fill=black] (0,0) circle (0.06);
        \draw[fill=black] (1,1) circle (0.06);
        \draw[fill=black] (2,0) circle (0.06);
        \draw[fill=white] (0,1) circle (0.06);
        \draw[fill=white] (1,0) circle (0.06);
        \draw[fill=white] (2,1) circle (0.06);

        \draw (1,-0.75) node {$a \, \det(B) \det(C)$};

        \begin {scope}[shift={(3,0)}]
            \draw (0,0) grid (2,1);

            \draw[red, line width=1.5] (0,0) -- (0,1);
            \draw[red, line width=1.5] (1,0) -- (2,0) -- (2,1) -- (1,1);
            \draw[red, line width=1.5] (1,-0.08) -- (2,-0.08);
            \draw[red, line width=1.5] (1,1.08) -- (2,1.08);
            
            \draw[fill=black] (0,0) circle (0.06);
            \draw[fill=black] (1,1) circle (0.06);
            \draw[fill=black] (2,0) circle (0.06);
            \draw[fill=white] (0,1) circle (0.06);
            \draw[fill=white] (1,0) circle (0.06);
            \draw[fill=white] (2,1) circle (0.06);
            
            \draw (1,-0.75) node {$a \, c_{33}$};
        \end {scope}
        
        \begin {scope}[shift={(6,0)}]
            \draw (0,0) grid (2,1);

            \draw[red, line width=1.5] (0,0) -- (1,0) -- (1,1) -- (0,1);
            \draw[red, line width=1.5] (2,0) -- (2,1);
            \draw[red, line width=1.5] (2.08,0) -- (2.08,1);
            \draw[red, line width=1.5] (2.16,0) -- (2.16,1);
            
            \draw[fill=black] (0,0) circle (0.06);
            \draw[fill=black] (1,1) circle (0.06);
            \draw[fill=black] (2,0) circle (0.06);
            \draw[fill=white] (0,1) circle (0.06);
            \draw[fill=white] (1,0) circle (0.06);
            \draw[fill=white] (2,1) circle (0.06);
            
            \draw (1,-0.75) node {$b_{22} \, \det(C)$};
        \end {scope}
        
        \begin {scope}[shift={(9,0)}]
            \draw (0,0) grid (2,1);
            
            \draw[red, line width=1.5] (0,0) -- (2,0) -- (2,1) -- (0,1);
            \draw[red, line width=1.5] (2.08,0) -- (2.08,1);
            
            \draw[fill=black] (0,0) circle (0.06);
            \draw[fill=black] (1,1) circle (0.06);
            \draw[fill=black] (2,0) circle (0.06);
            \draw[fill=white] (0,1) circle (0.06);
            \draw[fill=white] (1,0) circle (0.06);
            \draw[fill=white] (2,1) circle (0.06);
            
            \draw (1,-0.75) node {$\Delta_{1,3}^{1,3}(C)$};
        \end {scope}
        
        \begin {scope}[shift={(12,0)}]
            \draw (0,0) grid (2,1);

            \draw[red, line width=1.5] (0,0) -- (0,1);
            \draw[red, line width=1.5] (1,0) -- (2,0) -- (2,1) -- (1,1) -- cycle;
            \draw[red, line width=1.5] (2.08,0) -- (2.08,1);
            
            \draw[fill=black] (0,0) circle (0.06);
            \draw[fill=black] (1,1) circle (0.06);
            \draw[fill=black] (2,0) circle (0.06);
            \draw[fill=white] (0,1) circle (0.06);
            \draw[fill=white] (1,0) circle (0.06);
            \draw[fill=white] (2,1) circle (0.06);
            
            \draw (1.5,-0.75) node {$a \det(B)\det(C) \mathrm{tr}\left( \left(\begin{smallmatrix}B^{-1} & 0 \\ 0 & 0 \end{smallmatrix} \right) C^{-1} \right)$};
        \end {scope}
        
        \begin {scope}[shift={(0,-3)}]
            \draw (0,0) grid (2,1);

            \draw[red, line width=1.5] (0,0) -- (0,1);
            \draw[red, line width=1.5] (1,0) -- (2,0) -- (2,1) -- (1,1) -- cycle;
            \draw[red, line width=1.5] (2.08,0) -- (2.08,1);
            
            \draw[fill=black] (0,0) circle (0.06);
            \draw[fill=black] (1,1) circle (0.06);
            \draw[fill=black] (2,0) circle (0.06);
            \draw[fill=white] (0,1) circle (0.06);
            \draw[fill=white] (1,0) circle (0.06);
            \draw[fill=white] (2,1) circle (0.06);
            
            \draw (0,0) node[above left] {\scriptsize 1};
            \draw (0,1) node[below left] {\scriptsize 1};

            \draw (1,0) node[above left] {\scriptsize 2};
            \draw (1,1) node[below left] {\scriptsize 2};
            
            \draw (1,0) node[below right] {\scriptsize 1};
            \draw (2,0) node[below left] {\scriptsize 1};
            
            \draw (1,1) node[above right] {\scriptsize 1};
            \draw (2,1) node[above left] {\scriptsize 1};

            \draw (2,0) node[above right] {\scriptsize 2,3};
            \draw (2,1) node[below right] {\scriptsize 2,3};

            \draw (1,-0.75) node {$a \, b_{2,2} \Delta_{2,3}^{2,3}(C)$};
        \end {scope}
        
        \begin {scope}[shift={(3,-3)}]
            \draw (0,0) grid (2,1);

            \draw[red, line width=1.5] (0,0) -- (0,1);
            \draw[red, line width=1.5] (1,0) -- (2,0) -- (2,1) -- (1,1) -- cycle;
            \draw[red, line width=1.5] (2.08,0) -- (2.08,1);
            
            \draw[fill=black] (0,0) circle (0.06);
            \draw[fill=black] (1,1) circle (0.06);
            \draw[fill=black] (2,0) circle (0.06);
            \draw[fill=white] (0,1) circle (0.06);
            \draw[fill=white] (1,0) circle (0.06);
            \draw[fill=white] (2,1) circle (0.06);
            
            \draw (0,0) node[above left] {\scriptsize 1};
            \draw (0,1) node[below left] {\scriptsize 1};

            \draw (1,0) node[above left] {\scriptsize 1};
            \draw (1,1) node[below left] {\scriptsize 2};
            
            \draw (1,0) node[below right] {\scriptsize 2};
            \draw (2,0) node[below left] {\scriptsize 2};
            
            \draw (1,1) node[above right] {\scriptsize 1};
            \draw (2,1) node[above left] {\scriptsize 1};

            \draw (2,0) node[above right] {\scriptsize 1,3};
            \draw (2,1) node[below right] {\scriptsize 2,3};

            \draw (1,-0.75) node {$a \, b_{1,2} \Delta_{2,3}^{1,3}(C)$};
        \end {scope}
        
        \begin {scope}[shift={(6,-3)}]
            \draw (0,0) grid (2,1);

            \draw[red, line width=1.5] (0,0) -- (0,1);
            \draw[red, line width=1.5] (1,0) -- (2,0) -- (2,1) -- (1,1) -- cycle;
            \draw[red, line width=1.5] (2.08,0) -- (2.08,1);
            
            \draw[fill=black] (0,0) circle (0.06);
            \draw[fill=black] (1,1) circle (0.06);
            \draw[fill=black] (2,0) circle (0.06);
            \draw[fill=white] (0,1) circle (0.06);
            \draw[fill=white] (1,0) circle (0.06);
            \draw[fill=white] (2,1) circle (0.06);
            
            \draw (0,0) node[above left] {\scriptsize 1};
            \draw (0,1) node[below left] {\scriptsize 1};

            \draw (1,0) node[above left] {\scriptsize 2};
            \draw (1,1) node[below left] {\scriptsize 1};
            
            \draw (1,0) node[below right] {\scriptsize 1};
            \draw (2,0) node[below left] {\scriptsize 1};
            
            \draw (1,1) node[above right] {\scriptsize 2};
            \draw (2,1) node[above left] {\scriptsize 2};

            \draw (2,0) node[above right] {\scriptsize 2,3};
            \draw (2,1) node[below right] {\scriptsize 1,3};

            \draw (1,-0.75) node {$a \, b_{2,1} \Delta_{1,3}^{2,3}(C)$};
        \end {scope}
        
        \begin {scope}[shift={(9,-3)}]
            \draw (0,0) grid (2,1);

            \draw[red, line width=1.5] (0,0) -- (0,1);
            \draw[red, line width=1.5] (1,0) -- (2,0) -- (2,1) -- (1,1) -- cycle;
            \draw[red, line width=1.5] (2.08,0) -- (2.08,1);
            
            \draw[fill=black] (0,0) circle (0.06);
            \draw[fill=black] (1,1) circle (0.06);
            \draw[fill=black] (2,0) circle (0.06);
            \draw[fill=white] (0,1) circle (0.06);
            \draw[fill=white] (1,0) circle (0.06);
            \draw[fill=white] (2,1) circle (0.06);
            
            \draw (0,0) node[above left] {\scriptsize 1};
            \draw (0,1) node[below left] {\scriptsize 1};

            \draw (1,0) node[above left] {\scriptsize 1};
            \draw (1,1) node[below left] {\scriptsize 1};
            
            \draw (1,0) node[below right] {\scriptsize 2};
            \draw (2,0) node[below left] {\scriptsize 2};
            
            \draw (1,1) node[above right] {\scriptsize 2};
            \draw (2,1) node[above left] {\scriptsize 2};

            \draw (2,0) node[above right] {\scriptsize 1,3};
            \draw (2,1) node[below right] {\scriptsize 1,3};

            \draw (1,-0.75) node {$a \, b_{1,1} \Delta_{1,3}^{1,3}(C)$};
        \end {scope}
    \end {tikzpicture}
    \caption {(Top) A graph with vertex labels $\vec{n}$ and a choice of matrix edge weights. (Middle) The five mixed dimer covers of this graph, 
    with their weights given below each picture. (Bottom) The four half-edge colorings of the last mixed dimer cover. The weight of this configuration
    is the sum of the four contributions from these colorings.}
    \label {fig:mixed_ex}
    \end {figure}
    
\end {example} 

Mixed dimer covers are a common generalization of several combinatorial objects. When $n_v = 1$ at all vertices, then an $\vec{n}$-dimer cover
is just a perfect matching. If $n_v = n$ at all vertices, then it is an $n$-dimer cover, as discussed in the earlier sections. An interesting
example of the mixed case comes from ice-type statistical mechanics models, such as the six-vertex and twenty-vertex models, 
which we will describe in the next section.

Because the partition function is still given by a Kasteleyn determinant, we may again obtain the joint probability generating function of
all the edge multiplicities by introducing formal variables for the edges. If $t_e$ is a variable for each edge, replace each edge matrix $A$
by $t_e A$ in the block Kasteleyn matrix to obtain a new matrix $\widetilde{K}$. Then, as before, $\det(K^{-1} \widetilde{K})$ is the joint
probability generating function, where the $t_1^{i_1} \cdots t_N^{i_N}$ coefficient is the probability of the configuration where edge $e_j$ 
occurs with multiplicity $i_j$.

It is still possible to define a probability matrix $P_e = K^{[j],[i]}K_{[i],[j]}$ for each edge. All of the results in section \ref{sec:formulas} still
hold in this more general setting, since the proof of Theorem \ref{thm:pgf} does not require the matrices to be square. More specifically,
the matrix $K^{-1} \widetilde{K}$ has both rows and columns indexed by the black vertices, and so the diagonal blocks will be square matrices.

\subsection {The Six Vertex Model}

In the \emph{six vertex model} (also called \emph{square ice}), one starts with a subgraph of the square lattice $\Bbb{Z}^2$, with all vertices
colored black. We then add an extra white vertex in the center of each edge. A configuration in this model is an $\vec{n}$-dimer cover,
where $n_b = 2$ at all black vertices, and $n_w = 1$ at all white vertices. An example is illustrated in Figure \ref{fig:6v_domain_wall}. 
The local configurations are thought of as water molecules,
since each black (oxygen) vertex is connected to two of its white (hydrogen) neighbors. The name of the model comes from the fact that near
each black vertex, there are $\binom{4}{2} = 6$ possible ways to pair with two neighbors.

One then defines probability measures on the set of configurations parameterized by local \emph{Boltzmann weights}. To each of the six local
configurations, we assign six positive numbers $a_1,a_2,b_1,b_2,c_1,c_2$ (see Figure \ref{fig:boltzmann_weights}). 
The weight of each configuration is by definition the product of the
local weights over all the black vertices. The probability of each configuration is defined to be proportional to its weight.

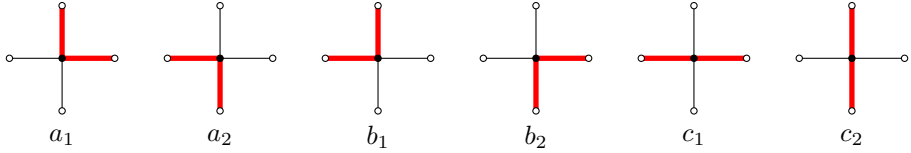
\begin {figure}
\centering
\begin {tikzpicture}[scale=0.7]
    \draw (0,-1) -- (0,1);
    \draw (-1,0) -- (1,0);

    \draw [red, line width=2] (1,0) -- (0,0) -- (0,1);
    \draw[fill=black] (0,0) circle (0.06);
    \draw[fill=white] (1,0) circle (0.06);
    \draw[fill=white] (-1,0) circle (0.06);
    \draw[fill=white] (0,1) circle (0.06);
    \draw[fill=white] (0,-1) circle (0.06);

    \draw (0,-1.5) node {$a_1$};

    \begin {scope}[shift={(3,0)}]
        \draw (0,-1) -- (0,1);
        \draw (-1,0) -- (1,0);
    
        \draw [red, line width=2] (-1,0) -- (0,0) -- (0,-1);
        \draw[fill=black] (0,0) circle (0.06);
        \draw[fill=white] (1,0) circle (0.06);
        \draw[fill=white] (-1,0) circle (0.06);
        \draw[fill=white] (0,1) circle (0.06);
        \draw[fill=white] (0,-1) circle (0.06);
    
        \draw (0,-1.5) node {$a_2$};
    \end {scope}
    
    \begin {scope}[shift={(6,0)}]
        \draw (0,-1) -- (0,1);
        \draw (-1,0) -- (1,0);
    
        \draw [red, line width=2] (0,1) -- (0,0) -- (-1,0);
        \draw[fill=black] (0,0) circle (0.06);
        \draw[fill=white] (1,0) circle (0.06);
        \draw[fill=white] (-1,0) circle (0.06);
        \draw[fill=white] (0,1) circle (0.06);
        \draw[fill=white] (0,-1) circle (0.06);
    
        \draw (0,-1.5) node {$b_1$};
    \end {scope}
    
    \begin {scope}[shift={(9,0)}]
        \draw (0,-1) -- (0,1);
        \draw (-1,0) -- (1,0);
    
        \draw [red, line width=2] (1,0) -- (0,0) -- (0,-1);
        \draw[fill=black] (0,0) circle (0.06);
        \draw[fill=white] (1,0) circle (0.06);
        \draw[fill=white] (-1,0) circle (0.06);
        \draw[fill=white] (0,1) circle (0.06);
        \draw[fill=white] (0,-1) circle (0.06);
    
        \draw (0,-1.5) node {$b_2$};
    \end {scope}
    
    \begin {scope}[shift={(12,0)}]
        \draw (0,-1) -- (0,1);
        \draw (-1,0) -- (1,0);
    
        \draw [red, line width=2] (1,0) -- (-1,0);
        \draw[fill=black] (0,0) circle (0.06);
        \draw[fill=white] (1,0) circle (0.06);
        \draw[fill=white] (-1,0) circle (0.06);
        \draw[fill=white] (0,1) circle (0.06);
        \draw[fill=white] (0,-1) circle (0.06);
    
        \draw (0,-1.5) node {$c_1$};
    \end {scope}
    
    \begin {scope}[shift={(15,0)}]
        \draw (0,-1) -- (0,1);
        \draw (-1,0) -- (1,0);
    
        \draw [red, line width=2] (0,1) -- (0,-1);
        \draw[fill=black] (0,0) circle (0.06);
        \draw[fill=white] (1,0) circle (0.06);
        \draw[fill=white] (-1,0) circle (0.06);
        \draw[fill=white] (0,1) circle (0.06);
        \draw[fill=white] (0,-1) circle (0.06);
    
        \draw (0,-1.5) node {$c_2$};
    \end {scope}
\end {tikzpicture}
\caption {Boltzmann weights for the six-vertex model.}
\label {fig:boltzmann_weights}
\end {figure}

The \emph{free fermionic subvariety} is a special subset of Boltzmann weights satisfying $c_1 c_2 = a_1a_2 + b_1b_2$.
It is well-known that in the free fermionic case, one can map the six vertex model to a dimer model (that is, for ordinary
perfect matchings rather than mixed dimer covers) \cite{elkp_92} \cite{fs_06}. There is also a general construction
in \cite{ko_23}, where any mixed dimer model may be mapped to a single dimer model. For the six-vertex model,
the free fermionic relation is equivalent to the Pl\"{u}cker relation in the Grassmannian $\mathrm{Gr}_2(4)$, and hence there is a $4 \times 2$
matrix whose maximal minors are the 6 Boltzmann weights. The rows of this matrix are the $1 \times 2$ matrices decorating the
4 edges around each black vertex in our model. Because of this, we will only obtain free fermionic weights through our
mixed $M_n$-dimer model.

\begin {figure}
\centering
\begin {tikzpicture}
    \draw[gray] (0,0) grid (2,2);
    
    \foreach \y in {0,1,2} {
        \draw (-0.5,\y) -- (0,\y);
    }

    \draw[red, line width=1.5] (-0.5,0) -- (0.5,0);
    \draw[red, line width=1.5] (-0.5,1) -- (0,1) -- (0,0.5);
    \draw[red, line width=1.5] (-0.5,2) -- (0,2) -- (0,1.5);
    
    \draw[red, line width=1.5] (1,0.5) -- (1,0) -- (1.5,0);
    \draw[red, line width=1.5] (0.5,1) -- (1,1) -- (1,1.5);
    \draw[red, line width=1.5] (0.5,2) -- (1.5,2);
    
    \draw[red, line width=1.5] (2,0.5) -- (2,0) -- (2.5,0);
    \draw[red, line width=1.5] (1.5,1) -- (2.5,1);
    \draw[red, line width=1.5] (2,1.5) -- (2,2) -- (2.5,2);
    
    \foreach \x in {0,1,2} {
        \foreach \y in {0,1,2} {
            \draw[fill=black] (\x,\y) circle (0.06);
        }
    }

    \foreach \y in {0,1,2} {
        \draw[fill=white] (-0.5,\y) circle (0.06);
        \draw[fill=white] (2.5,\y) circle (0.06);

        \foreach \x in {0,1} {
            \draw[fill=white] (\x + 0.5, \y) circle (0.06);
            \draw[fill=white] (\y, \x + 0.5) circle (0.06);
        }
    }
\end {tikzpicture}
\caption {A configuration of the six-vertex model with domain wall boundary conditions.}
\label {fig:6v_domain_wall}
\end {figure}
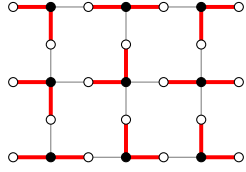

\begin {example}
    Let us suppose the Boltzmann weights are symmetric, in the sense that $a_1 = a_2$, $b_1 = b_2$, and $c_1 = c_2$. Then, up to gauge transformations,
    we can parameterize the free-fermionic weights by a single parameter $\theta$, where $a = \sin(\theta)$, $b = \cos(\theta)$, and $c = 1$. In our
    dimer model with matrix edge weights, this corresponds to the four edges around each black vertex being decorated by the following four $1 \times 2$ matrices:
    \[ v_E = \left( 1,0 \right), \quad v_N = \left( \cos(\theta), \sin(\theta)\right), \quad v_W = (0,1), \quad v_S = \left(-\sin(\theta), \cos(\theta)\right), \]
    where $E,N,W,S$ correspond to the four cardinal directions.

    One can then use the results of Section~\ref{sec:formulas} to compute local edge probabilities. For example, consider a $3 \times 3$ grid
    with \emph{domain wall boundary conditions} as in Figure \ref{fig:6v_domain_wall}. This means we have the extra 1-valent white vertices on the
    left and right sides (but not the top and bottom). Its Kasteleyn matrix is an $18 \times 18$ matrix consisting of $1 \times 2$ blocks, each of which
    is one of the vectors $v_E$, $v_W$, $v_N$, or $v_S$. 

    Let us consider the central black vertex. Abbreviating $c = \cos(\theta)$ and $s = \sin(\theta)$, 
    the blocks of the inverse of the Kasteleyn matrix, corresponding to the adjacent edges in the $E$, $N$, $W$, and $S$
    directions respectively are
    \[ 
        \begin{pmatrix} c^4 + s^4 \\ cs^2 ( s^2 - c^2 ) \end{pmatrix}, \quad  
        \begin{pmatrix} cs^2 \\ c^2 s \end{pmatrix}, \quad 
        \begin{pmatrix} cs^2(c^2-s^2) \\ c^4+s^4 \end{pmatrix}, \quad
        \begin{pmatrix} -c^2s \\ cs^2 \end{pmatrix}
    \]
    From this, we can conclude that the probability of the east edge (and also the west) being used is $\cos^4(\theta) + \sin^4(\theta)$.
    Similarly, the probability of the north edge (and also the south) being used is $2\cos^2(\theta)\sin^2(\theta) = \frac{1}{2} \sin^2(2\theta)$.
    A particularly symmetric choice is $\theta = \pi/4$, in which case all four probabilities are $1/2$.
\end {example}

\section*{Acknowledgments}

We would like to thank University of Minnesota for their hospitality during the MRWAC research workshop, where this project began.
The workshop was supported by the NSF grants DMS-1745638 and DMS-1854162.

\vfill

\bibliographystyle{alpha}
\bibliography{main}

\end{document}